\documentclass{article} 
\usepackage{amssymb}
\usepackage{amscd}
\usepackage{amsmath} 
\usepackage[dvipdfmx]{graphicx}
\usepackage{latexsym} 
\usepackage{enumerate}
\usepackage[usenames]{color}
\usepackage{lscape}  
\allowdisplaybreaks

\usepackage{theorem}
\theoremstyle{plain}
\theorembodyfont{\itshape}
\newtheorem{theorem}{Theorem}[section]
\newtheorem{proposition}[theorem]{Proposition}
\newtheorem{lemma}[theorem]{Lemma}
\newtheorem{corollary}[theorem]{Corollary}

\theorembodyfont{\rmfamily}

\newtheorem{remark}[theorem]{Remark}

\newtheorem{example}[theorem]{Example}

\newtheorem{problemgpf}{Problem I}

\newtheorem{problemocf}{Problem I\!I}

\makeatletter
\@addtoreset{figure}{section}
\@addtoreset{table}{section}
\makeatother

\def\re{\mathrm{e}}
\def\ri{\mathrm{i}}
\def\rI{\mathrm{I}}
\def\rII{\mathrm{I\!I}}
\def\rIm{\mathrm{Im}}
\def\rRe{\mathrm{Re}} 
\def\bC{\mathbb{C}}
\def\bD{\mathbb{D}}
\def\bN{\mathbb{N}}
\def\bQ{\mathbb{Q}}
\def\bR{\mathbb{R}}
\def\bZ{\mathbb{Z}}
\def\cD{\mathcal{D}}
\def\cF{\mathcal{F}}
\def\cI{\mathcal{I}}
\def\cJ{\mathcal{J}} 
\def\ba{\mbox{\boldmath $a$}}
\def\be{\mbox{\boldmath $e$}}
\def\bmf{\mbox{\boldmath $f$}}
\def\bmF{\mbox{\boldmath $F$}}
\def\bp{\mbox{\boldmath $p$}}

\def\bv{\mbox{\boldmath $v$}}
\def\bal{\mbox{\boldmath $\alpha$}}
\def\1{\mbox{\boldmath $1$}}
\def\vG{\varGamma}
\def\hgF{{}_2F_1}
\def\hgG{{}_2G_1}
\def\hgH{{}_2H_1}
\def\ds{\displaystyle}
\def\ts{\textstyle}
\title{\bf Arithmetic Constraints on Hypergeometric 
Identities\thanks{ This work was supported by JSPS KAKENHI Grant Number JP22K03365.}}  
\author{Katsunori Iwasaki\thanks{Professor Emeritus, Hokkaido University, 
Sapporo 060-0810 Japan. E-mail: {\tt iwasaki@math.sci.hokudai.ac.jp}} \, 
and Mina Kusakabe\thanks{Passed away while this work was in progress.}}
\date{}  
\begin{document}
\maketitle
\begin{abstract}
The standard literature on special functions contains a lot of hypergeometric 
identities involving products and quotients of gamma functions, but still 
the occurrence of such identities is a sporadic phenomenon.   
This is because the existence of them is constrained by severe  
arithmetic conditions.   
We demonstrate this kind of constraints by focusing on a certain data region 
where the essential nature of the issue comes out clearly.      
\end{abstract}  
\section{Introduction} \label{sec:intro}
Let $\hgF(\alpha, \beta; \gamma; z)$ be the Gauss hypergeometric function 
and $\vG(w)$ the Euler gamma function.  
Given a sextuple of data $\lambda = (p, q, r; a, b; x) \in \bR^6$ with $0 < x < 1$, 
we set $f(w; \lambda) := \hgF(p w + a, q w + b; r w; x)$.     
As a continuation to our previous studies in \cite{Iwasaki1,Iwasaki2}, 
we are interested in the following problem. 
\begin{problemgpf} \label{Prob-I}
Find or characterize those data $\lambda$ which give rise to 
a {\sl gamma product formula} (GPF for short)      
\begin{equation} \label{eqn:GPF}
f(w; \lambda) = C \cdot d^w \cdot 
\dfrac{\vG(w + u_1) \cdots \vG(w + u_m)}{\vG(w + v_1) \cdots \vG(w + v_n)}, 
\qquad w \in \bC,    
\end{equation}
for some constants $C$, $d \in \bC^{\times}$ and $u_1, \dots, u_m$, $v_1, \dots, v_n \in \bC$,       
where $w$ is a free complex parameter. 
\end{problemgpf}
\par
The standard literature on special functions, e.g. \cite{AAR,Brychkov,Ebisu2,Erdelyi,OLBC}, 
contains a lot of identities of the form \eqref{eqn:GPF}, but still they occur only rarely  
among all data $\lambda = (p, q, r; a, b; x)$ which vary continuously.    
The aim of this article is to address the question of why GPFs are so 
sporadic, by observing that solutions $\lambda$ to Problem I must be subject to certain 
{\sl arithmetic constraints}, typically (but not always) like the condition    
that $p, q, r$ must be integers, $a, b$ must be rational numbers,  
and $x$ must be an algebraic number of a special form.         
\par
\begin{figure}[h]
\centerline{\includegraphics*[width=42mm,clip]{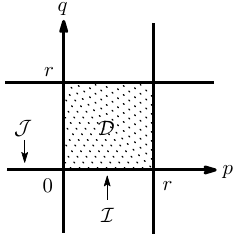}}
\vspace{-1mm}
\caption{Partition of the $(p, q)$-plane with a fixed $r > 0$.} 
\label{fig:square} 
\end{figure}
How GPFs come about depends strongly on the direction 
of the vector $\bp = (p, q, r) \in \bR^3$, which we call the {\sl principal part} of $\lambda$. 
This is a kind of Stokes phenomenon with respect to the large parameter $w$. 
In \cite{Iwasaki1,Iwasaki2}, with a fixed $r > 0$, the $(p, q)$-plane is partitioned into nine 
domains and their boundaries as indicated in Figure \ref{fig:square}.  
Our main focus in \cite{Iwasaki1,Iwasaki2} was on the central square $\cD$, the shaded area   
in Figure \ref{fig:square}. 
In this article we turn our attention to its {\sl boundary}.  
Although the boundary is thinner and more special than the open square,  
but exactly because of that, {\sl arithmetic nature} of Problem I comes out 
more vividly on the boundary.  
\par
Given $k \in \bN := \bZ_{\ge 1}$, the $k$-{\sl multiple} of $\lambda  = (p, q, r; a, b; x)$ 
is defined by $k \lambda := (k p, k q, k r; a, b; x)$. 
It is obvious from the definition of $f(w; \lambda)$ that $f(w; k \lambda) = f(k w; \lambda)$.  
This identity and Gauss's multiplication formula 
\begin{equation} \label{eqn:GMF}
\vG(k w) = (2 \pi)^{(1-k)/2} \cdot k^{k w - 1/2} \cdot 
\prod_{j=0}^{k-1} \vG\big(w + \ts \frac{j}{k} \big), \qquad k \in \bN 
\end{equation}
imply that any multiple of a solution to Problem I is again a solution; see \cite[\S1]{Iwasaki2} for details.        
A solution is said to be {\sl primitive} if it is not a multiple of any other solution. 
We say that $\lambda$ is {\sl elementary} if $f(w; \lambda)$   
has at most finitely many poles in $\bC_{w}$. 
We can then speak of elementary or non-elementary solutions to Problem I.        
\par
The hypergeometric function $\hgF(\alpha, \beta; \gamma; z)$ has two basic symmetries 
\begin{equation} \label{eqn:euler}
\hgF(\alpha, \beta; \gamma; z) = \hgF(\beta, \alpha; \gamma; z), \qquad  
\hgF(\alpha, \beta; \gamma; z) = (1-z)^{\gamma-\alpha-\beta } 
\hgF(\gamma-\alpha, \gamma-\beta; \gamma; z). 
\end{equation}    
The former identity is obvious from the definition of $\hgF(\alpha, \beta; \gamma; z)$, 
while the latter is known as Euler's transformation.   
They induce transformations of data $\lambda$ which takes one solution to another,   
\begin{equation} \label{eqn:bsym} 
\lambda = (p, q, r; a, b; x) \mapsto (q, p, r; b, a; x), \qquad 
\lambda = (p, q, r; a, b; x) \mapsto (r-p, r-q, r; -a, -b; x).  
\end{equation}
The boundary of the central square $\cD$ consists of four sides, 
but a combination of symmetries \eqref{eqn:bsym} brings any side 
to the {\sl south side} $\cI$ in Figure \ref{fig:square}. 
The full parameter region corresponding to the south side is given by           
\begin{equation*} \label{eqn:cI}
\cI := \{ \lambda = (p, q, r; a, b; x) \in\bR^6 \mid 0 < p < r, \,\, q = 0, \,\, 0 < x < 1 \},         
\end{equation*}
which is denoted by the same symbol $\cI$. 
Our main concern in this article is the study of Problem I in this region. 
As a part of our earlier result \cite[Theorem 2.4]{Iwasaki1} the following 
constraint is already known to us.  
\begin{theorem} \label{thm:recall}
Let $\lambda = (p, 0, r; a, b; x) \in \cI$.  
\begin{enumerate}
\setlength{\itemsep}{-1pt}
\item $\lambda$ is an elementary solution to Problem I if and only if $b$ 
is zero or a negative integer, in which case $f(w; \lambda)$ is a rational function of $w$.      
\item If $\lambda$ is a non-elementary solution to Problem I, then   
\begin{equation} \label{eqn:integral} 
p, \, r \in \bZ; \qquad p \equiv r \bmod 2; \qquad b = 1/2.  
\end{equation}
In particular, any non-elementary solution is a multiple of a primitive solution.  
\end{enumerate} 
\end{theorem}
\par
The aim of this article is to strengthen Theorem \ref{thm:recall} to a large extent 
by establishing the main results stated in \S \ref{sec:results}.  
This goal is achieved by making a careful examination of contiguous relations 
for hyergeometric functions and by exploiting the concepts of {\sl duality} and {\sl reciprocity} 
introduced in \cite[Definitions 1.2 and 1.3]{Iwasaki2}, which are the transformations of data 
$\lambda \mapsto \lambda' = (p',q',r';a', b'; x')$ and 
$\lambda \mapsto \check{\lambda} = (\check{p}, \check{q}, \check{r}; \check{a}, \check{b}; \check{x})$ 
defined by  
\begin{subequations}
\begin{gather}
p' := p, \quad q' := q, \quad r' := r; \quad 
a' := 1- \frac{2 p}{r}-a, \quad b' := 1- \frac{2 q}{r}-b; \quad x' := x, 
\label{eqn:duality} \\[1mm]
\begin{split}
\check{p} &:= -p, \qquad \check{q} := -q, \qquad \check{r} := r-p-q;  \qquad \check{x} := 1-x;  
\\[1mm]
& \check{a} := \dfrac{(r-q)(1-a) - p b}{r-p-q}, \qquad \check{b} := \dfrac{(r-p)(1-b) - qa}{r-p-q}. 
\end{split} 
\label{eqn:recip} 
\end{gather}
\end{subequations} 
Notice that duality defines an involution on $\cI$, while reciprocity provides 
a bijection of $\cI$ onto  
\begin{equation*} \label{eqn:cJ} 
\cJ := \{ \lambda = (p, q, r; a, b; x) \in \bR^6 \mid p < 0, \,\, q = 0, \,\, 0 < r, \,\, 0 < x < 1 \}, 
\end{equation*}
which corresponds to the half-line $\cJ$ in Figure \ref{fig:square}.   
Notice also that for $\lambda \in \cI$ if $b = 1/2$ then $b' = \check{b} = 1/2$. 
\section{Main Results} \label{sec:results} 
We state the main results of this article and give a couple of examples to illustrate them.   
\begin{theorem} \label{thm:a-x}
Let $\lambda = (p, 0, r; a, 1/2; x) \in \cI$ be any non-elementary solution to Problem I.   
\begin{enumerate}
\setlength{\itemsep}{-1pt}
\item The dual $\lambda' = (p, 0, r; a', 1/2; x) \in \cI$ to the solution $\lambda$ is 
also a non-elementary solution to Problem I.  
\item The ratio $s := r/p$ is an integer with $s \ge 2$ and $p (s-1) \equiv 0 \bmod 2$.  
There exist $j$, $j' \in \bZ_{\ge 0}$ such that 
\begin{equation} \label{eqn:saap}
a = j/s, \qquad a' = j'/s, \qquad j + j' = s-2.   
\end{equation}
In particular, $a$ and $a'$ are rational numbers.    
\item The argument $x$ is the unique root in the interval $(0, \, 1)$ of 
the algebraic equation 
\begin{equation} \label{eqn:alg-eq}
\varphi_s(z) := (s-1)^{s-1} z^s -s^s (1-z)^{s-1} = 0. 
\end{equation}
In particular, $x$ is an algebraic number.   
\end{enumerate}
\end{theorem}
\par
Certainly, equation \eqref{eqn:alg-eq} has a unique root $x \in (0, \, 1)$,    
since $\varphi_s(0) = -s^s < 0 < \varphi_s(1) = (s-1)^{s-1}$ and 
$\varphi_s'(z) = s (s-1) \{ (s-1)^{s-2} z^{s-1} + s^{s-1} (1-z)^{s-2} \} > 0$ for $z \in (0, \, 1)$. 
Theorem \ref{thm:a-x} implies that, once the ratio $s = r/p$ is given,  
there are at most $s-1$ possibilities for the triple $(a, a', x)$.  
The theorem does not claim that any expression for $(a, a')$ in 
\eqref{eqn:saap} leads to a solution; some or all of those candidates 
may fail to give a solution. 
\begin{theorem} \label{thm:gpf} 
Let $\lambda = (p, 0, r; a, 1/2; x) \in \cI$ be any non-elementary solution to Problem I. 
\begin{enumerate} 
\setlength{\itemsep}{-1pt}
\item There exists a factorization into two components of the right-hand side below   
\begin{equation} \label{eqn:v*}
\prod_{i=1}^r (w+v_i) \prod_{i=1}^r (w+v_i^*) =  
\prod_{i=0}^{p-1}\Big(w+ \ts\frac{i+a}{p} \Big) 
\ds \prod_{j=0}^{r-p-1} \Big( w + \ts\frac{j-a}{r-p} \Big) 
\ds \prod_{j=0}^{r-1} \Big( w + \ts\frac{j-1/2}{r} \Big)   
\end{equation}
such that the solution $\lambda$ and its dual $\lambda' \in \cI$ admit GPFs of the forms   
\begin{subequations} \label{eqn:gpf}
\begin{align}
f(w; \lambda) 
&= C \cdot \dfrac{\prod_{i=0}^{r-1} \vG(w+ \frac{i}{r})}{\prod_{i=1}^r \vG(w + v_i)}, 
\label{eqn:gpf1} \\[1mm]
f(w; \lambda') 
&= C' \cdot \dfrac{\prod_{i=0}^{r-1} \vG(w+ \frac{i}{r})}{\prod_{i=1}^r \vG(w + v_i')}, 
\qquad v_i' := 1 - \frac{2}{r} - v_i^*,  
\label{eqn:gpf2}  
\end{align}
\end{subequations}
where $C$ and $C'$ are nonzero constants. 
Moreover $v_1, \dots, v_r$ and $v_1', \dots, v_r'$ sum to the same value      
\begin{equation} \label{eqn:sum}  
v_1 + \cdots + v_r = v_1' + \cdots + v_r' = (r-1)/2.    
\end{equation}
\item There exist a division relation    
\begin{equation} \label{eqn:division2}
\prod_{i=0}^{p-1} \big(w+ \ts\frac{i+a}{p} \big) \, \Big| \, \ds \prod_{i=1}^r (w+v_i),   
\end{equation} 
which allows us to arrange $v_1, \dots, v_r$ so that 
$\prod_{i=0}^{p-1}\big(w + \ts \frac{i+a}{p} \big) = \prod_{i=r-p+1}^r (w + v_i)$.    
With this convention the reciprocal $\check{\lambda} \in \cJ$ to the 
solution $\lambda$ is also a solution to Problem I having GPF 
\begin{equation} \label{eqn:gpf-r}
f(w; \check{\lambda}) = \check{C} \cdot \dfrac{\prod_{i=0}^{r-p-1} 
\vG\big(w + \frac{i}{r-p}\big)}{ \prod_{i=1}^{r-p} \vG( w + \check{v}_i) }, 
\qquad \check{v}_i := v_i  - \frac{1/2 -a}{r-p},  
\end{equation}
where $\check{C}$ is a nonzero constant. 
\end{enumerate} 
\end{theorem} 
\par
The dilation constant $d$ in the general template \eqref{eqn:GPF} is just $1$ 
in GPFs \eqref{eqn:gpf} and \eqref{eqn:gpf-r}. 
The fractions in these formulas may be reducible, that is, 
they may have common factors in their numerators and denominators. 
\begin{theorem} \label{thm:degree} 
Given any non-elementary solution $\lambda =(p, 0, r; a, 1/2; x) \in \cI$ 
to Problem I, let $\deg x$ denote the degree of $x$ as an algebraic number. 
If the ratio $s := r/p$ is $2$ then $\deg x = 2$, whereas if $s \ge 3$ then 
\begin{equation*} \label{eqn:lb}
\deg x > \delta(s) := \frac{(\log p_s) (s-1)}{1 + \log(s-1)},  
\end{equation*}
where $p_s$ is the least prime factor of the integer $s-1$. 
In particular, once an upper bound for $\deg x$ is given, there 
remain only a finite number of possibilities for the ratio $s = r/p$.   
\end{theorem}
\begin{example} \label{ex:rat-x}   
There are two non-elementary solutions $\lambda \in \cI$ 
such that $x$ is a {\sl rational} number,              
\begin{subequations} \label{eqn:sol-r}
\begin{alignat}{2}
\lambda &= (1, \, 0, \, 3; \, 0, \, 1/2; \, 3/4), \qquad &
f(w; \lambda) &= \frac{2}{\sqrt{3}} \cdot  
\dfrac{\vG(w + 1/3) \, \vG(w+2/3)}{\vG( w + 1/2) \, 
\vG( w + 1/2)},  
\label{eqn:sol-r1}
\\[2mm]
\lambda &= (1, \, 0, \, 3; \, 1/3, \, 1/2; \, 3/4), 
\qquad &f(w; \lambda) &= \frac{2}{\sqrt{3}} \cdot 
\dfrac{\vG( w) \, \vG( w+2/3)}{\vG(w + 1/6) \, \vG(w + 1/2)},  
\label{eqn:sol-r2} 
\end{alignat}
\end{subequations} 
which are formulas (1,3,3-1,iii) and (1,3,3-1,xx) in Ebisu \cite{Ebisu2}. 
These solutions are dual to each other with $p = 1$, $s = r/p = 3$, 
$j = 0$ and $j' = 1$ in formula \eqref{eqn:saap}. 
Obviously, they are primitive, since $p = 1$.  
\end{example} 
\begin{example} \label{ex:irr-x}  
There is a non-elementary solution $\lambda \in \cI$ such that $x$ is a 
{\sl quadratic irrational},         
\begin{equation} \label{eqn:sol-ir}
\lambda = (2, \, 0, \, 4; \, 0, \, 1/2; \, 2 \sqrt{2}-2),  
\qquad 
f(w; \lambda) = \frac{1}{\sqrt{ 2-\sqrt{2} }} \cdot    
\dfrac{ \vG\left( w+1/4 \right) 
\vG( w+3/4 )}{\vG( w+3/8) 
\vG( w+5/8)},  
\end{equation} 
which is formula (2,4,4-2,ii) in \cite{Ebisu2}. 
This solution is self-dual with $s = r/p = 2$, $j = j' = 0$ in formula \eqref{eqn:saap}.  
Conversely, for any solution with $r/p = 2$, formula \eqref{eqn:saap} forces 
$j = j' = 0$ and $p \equiv 0 \bmod 2$, so it must be a multiple of solution \eqref{eqn:sol-ir}. 
Thus solution \eqref{eqn:sol-ir} is primitive and there is no solution 
with principal part $\bp = (1, 0, 2)$. 
\end{example}
\begin{example} \label{ex:reciprocal}    
The reciprocals to solutions \eqref{eqn:sol-r1} and 
\eqref{eqn:sol-r2} are given respectively by 
\begin{subequations} \label{eqn:rsol-r}
\begin{alignat}{2}
\lambda &= (-1, \, 0, \,  2; \, 5/4, \, 1/2; \, 1/4), 
\qquad &
f(w; \lambda) &= \frac{2 \sqrt{2}}{3} \cdot  
\dfrac{\vG(w) \, \vG(w+1/2)}{\vG( w + 1/4) \, \vG( w + 1/4)},  
\label{eqn:rsol-r1}
\\[2mm]
\lambda &= (-1, \, 0, \, 2; \, 3/4, \, 1/2; \, 1/4), 
\qquad &f(w; \lambda) &= \frac{2 \sqrt{2}}{3} \cdot 
\dfrac{\vG( w) \, \vG( w+ 1/2)}{\vG(w + 1/12) \, \vG(w + 5/12)},  
\label{eqn:rsol-r2} 
\end{alignat}
\end{subequations} 
which are formulas (1,3,3-1,xvi) and (1,3,3-1,xi) in Ebisu \cite{Ebisu2}.  
The reciprocal to solution \eqref{eqn:sol-ir} is given by  
\begin{equation} \label{eqn:rsol-ir}
\lambda = (-2, \, 0, \, 2; \, 3/2, \, 1/2; \, 3-2 \sqrt{2}),  
\qquad 
f(w; \lambda) = \frac{1}{\sqrt{ 2(2-\sqrt{2})}} \cdot    
\dfrac{ \vG(w) \, \vG\left( w + 1/2 \right)}{\vG\left( w + 1/8 \right) 
\vG\left( w + 3/8 \right)},  
\end{equation} 
which is formula (2,4,4-2,vi) in \cite{Ebisu2}. 
Solutions \eqref{eqn:rsol-r1} and \eqref{eqn:rsol-r2} are dual to each 
other, while \eqref{eqn:rsol-ir} is self-dual.   
\end{example}
\par
Theorems \ref{thm:a-x} and \ref{thm:degree} lead to the following result 
regarding Examples \ref{ex:rat-x} and \ref{ex:irr-x}.   
\begin{theorem} \label{thm:rqa}  
Consider non-elementary solutions $\lambda = (p, 0, r; a, 1/2; x) \in \cI$ to Problem I. 
\begin{enumerate}
\setlength{\itemsep}{-1pt}
\item Solutions \eqref{eqn:sol-r1} and \eqref{eqn:sol-r2} are the only primitive 
solutions such that $x$ is a rational number.  
\item Solution \eqref{eqn:sol-ir} is the only primitive solution   
such that $x$ is a quadratic irrational.   
\end{enumerate}  
\end{theorem} 
\par
The plan of this article is as follows. 
In \S \ref{sec:contig} Problem I is reduced to another problem, Problem 
$\rII$, which is closely related to contiguous relations for hypergeometric 
functions. 
So \S\ref{sec:contig} discusses various issues on contiguous relations: 
contiguous matrices in \S \ref{ss:m-contig}, their leading asymptotics 
in \S \ref{ss:p-p}, truncated hypergeometric products in \S \ref{ss:trunc}, 
and terminating hypergeometric sums in \S \ref{ss:thgs}. 
Theorems \ref{thm:a-x} and \ref{thm:gpf} are established in 
\S \ref{sec:d-r} by exploiting the concepts of duality and reciprocity, 
where duality is considered first in \S \ref{ss:dual}, reciprocity is discussed 
next in \S \ref{ss:r-cf} in conjunction with connection formula, and  
duality is re-examined lastly in more depth in \S \ref{ss:dual-r}. 
One more idea is added in \S \ref{ss:nsc} to make the conditions in 
Theorem \ref{thm:a-x} into a necessary and sufficient condition for 
$\lambda \in \cI$ to be a non-elementary solution; 
see Theorem \ref{thm:nsc}.     
The algebraic equation \eqref{eqn:alg-eq} satisfied by the argument $x$ 
is studied in \S \ref{sec:deg}.  
Theorem \ref{thm:degree} is established at the end of \S \ref{ss:arith}.  
Theorem \ref{thm:rqa} follows from Theorems \ref{thm:a-x} and \ref{thm:degree}.    
\section{Contiguous Relations} \label{sec:contig}
If $\lambda$ is a solution to Problem I with GPF \eqref{eqn:GPF}, then  
the recursion formula $\vG(w+1) = w \vG(w)$ yields   
\begin{equation} \label{eqn:CF}
\dfrac{f(w+1; \lambda)}{f(w; \lambda)} = 
d \cdot \dfrac{(w + u_1) \cdots (w + u_m)}{(w + v_1) \cdots (w + v_n)},  
\end{equation}
which is a rational function of $w$. 
This observation leads us to the following problem. 
\begin{problemocf} \label{Prob-II} 
Find or characterize those deta $\lambda$ for which 
$f(w+1; \lambda)/f(w; \lambda)$ is a rational function of $w$.   
\end{problemocf}
\par
We say that a piece of data $\lambda = (p, q, r; a, b; x)$ is {\sl integral} 
if its principal part $\bp = (p, q, r)$ is an integer vector. 
There is a general method to obtain integral solutions to Problem I\!I by means of 
contiguous relations; see \cite{Ebisu1,Iwasaki1}.   
For simplicity of notation we write $\hgF(\ba; z) := \hgF(\alpha, \beta; \gamma; z)$ 
for $\ba = (a_1, a_2, a_3) = (\alpha, \beta, \gamma)$ and $\1 := (1, 1, 1)$. 
Given an integer vector $\bp = (p, q, r) \in \bZ^3$, a repeated application 
of fifteen contiguous relations of Gauss gives rise to rational functions $r(\ba; \bp; z)$ 
and $q(\ba; \bp; z)$ of $(\ba; z)$ depending on $\bp$ such that 
\begin{equation} \label{eqn:ttr}
\hgF(\ba + \bp; z) = r(\ba; \bp; z) \, \hgF(\ba; z) + q(\ba; \bp; z) \, \hgF(\ba + \1; z).  
\end{equation}
Given a piece of integral data $\lambda = (p, q, r; a, b; x)$, we put 
$\widetilde{f}(w; \lambda) := \hgF(p w + a+ 1, q w + b + 1; r w +1; x)$. 
Substituting $\ba = (p w + a, q w + b, r w)$ and $z = x$ into the three-term 
relation \eqref{eqn:ttr}, we have   
\begin{equation} \label{eqn:ttr2}
f(w + 1; \lambda) = R(w; \lambda) \, f(w; \lambda) + Q(w; \lambda) \, \widetilde{f}(w; \lambda),  
\end{equation}
where $R(w; \lambda)$ and $Q(w; \lambda)$ are rational functions of $w$ depending on $\lambda$. 
If $Q(w; \lambda)$ {\sl happens to vanish identically}, then the three-term relation 
\eqref{eqn:ttr2} reduces to the two-term one $f(w+1; \lambda)/f(w; \lambda) = R(w; \lambda)$, 
leading to an integral solution $\lambda$ to Problem I\!I. 
A solution arising in this manner is said to {\sl come from contiguous relations}. 
\par
Now we bring our attention to the region $\cI$. 
The following result is known as a part of \cite[Theorem 5.4]{Iwasaki1}.    
\begin{theorem} \label{thm:lift} 
If $\lambda \in \cI$ is a solution to Problem I\!I and the ratio 
$f(w+1; \lambda)/f(w; \lambda)$ is expressed in the form \eqref{eqn:CF},  
then $\lambda$ is a solution to Problem I with GPF \eqref{eqn:GPF}. 
Moreover in formula \eqref{eqn:CF} we must have $m = n$ and  
\begin{equation} \label{eqn:sum2} 
u_1 + \cdots + u_m = v_1 + \cdots + v_m. 
\end{equation} 
\end{theorem}
\par
This is a consequence of asymptotic analysis of the function $f(w; \lambda)$ 
as $w$ tends to infinity.    
Thanks to Theorem \ref{thm:lift}, Problems I and I\!I are equivalent in the region $\cI$, 
so we can speak of a solution therein without specifying to which problem it is a solution.  
Recall from Theorem \ref{thm:recall} that any non-elementary solution to either problem 
in $\cI$ is integral. 
Together with these facts, the following result is also known in \cite[Theorem 2.4]{Iwasaki1}. 
\begin{theorem} \label{thm:cfcr} 
Every non-elementary solution in $\cI$ comes from contiguous relations. 
Therefore a piece of data $\lambda = (p, 0, r; a, 1/2; x)\in \cI$ is a solution if and only if 
$\lambda$ is integral and $Q(w; \lambda)$ vanishes identically. 
\end{theorem}
\par
It is natural to expect that constraints for the existence of non-elementary 
solutions originate from the properties of contiguous relations.   
Thus we begin our discussions with a closer look at contiguous relations.    
\subsection{Contiguous Matrices} \label{ss:m-contig} 
It is better to use the matrix version of contiguous relations.   
Let $\bmF(\ba) := {}^t (\hgF(\ba; z), \, \hgF(\ba + \1 ; z ))$. 
Then the three-term relation \eqref{eqn:ttr} and the one with $\bp$ replaced by 
$\bp + \1$ form the matrix equation   
\begin{equation} \label{eqn:m-3tr}
\bmF(\ba + \bp) = A(\ba ; \bp) \, \bmF(\ba), 
\qquad A (\ba ;\bp) := 
\begin{pmatrix}
r (\ba ;\bp; z) &  
q (\ba ;\bp; z)  \\[1mm]
r (\ba ;\bp + \1; z) & 
q (\ba ;\bp + \1; z)
\end{pmatrix}.   
\end{equation}
For $i = 1, 2, 3$, we put $A_i(\ba) := A(\ba; \be_i)$, where $\be_1 = (1, 0; 0)$, 
$\be_2 = (0,1,0)$ and $\be_3 = (0, 0; 1)$. 
These matrices satisfy the compatibility condition 
$A_i(\ba + \be_j) \, A_j(\ba) = A_j(\ba + \be_i) \, A_i(\ba)$. 
Given an integer vector $\bp \in \bZ_{\ge 0}^{3}$, if we take a sequence of indices 
$i = (i_1, \dots , i_k) \in \{1, 2, 3 \}^k$ such that $\bp = \be_{i_1} + \cdots + \be_{i_k}$, 
then we have    
\begin{equation} \label{eqn:mp}
A(\ba; \bp) 
= A_{i_{k}} (\ba + \be_{i_1} + \dots + \be_{i_{k-1}}) \cdots 
A_{i_{3}}(\ba + \be_{i_1} + \be_{i_2}) \, 
A_{i_{2}}(\ba + \be_{i_1}) \, A_{i_{1}}(\ba).    
\end{equation}
Due to compatibility condition, the matrix product on the right is independent of 
the choice of $i = (i_1, \dots , i_k)$. 
\par
In this article we take $\bp = (p, 0, r)$ and $i = (i_1, \dots, i_k) \in \{1, 3 \}^k$.  
We use the explicit formulas    
\begin{subequations} \label{eqn:A1-A3}
\begin{alignat}{2} 
A_1(\ba) &=  
\begin{pmatrix}
1  &  \frac{\beta z}{\gamma}  \\[2mm]
-\frac{\gamma}{(\alpha+1)(z-1)} & 
\frac{\gamma-\alpha-1-\beta z}{(\alpha+1)(z - 1)}
\end{pmatrix}, \qquad & 
\det A_1(\ba) &= \frac{\gamma - \alpha - 1}{(\alpha + 1)(z - 1)}, 
\label{eqn:A1} \\[2mm]
A_3(\ba) &=   
\begin{pmatrix}
\frac{\gamma(\gamma-\alpha-\beta)}{(\gamma-\alpha)(\gamma-\beta)} & 
- \frac{\alpha \beta (z - 1)}{(\gamma - \alpha)(\gamma - \beta)}  \\[2mm]
\frac{\gamma (\gamma + 1)}{(\gamma - \alpha)(\gamma - \beta)z} & 
\frac{\gamma (\gamma + 1)(z - 1)}{(\gamma - \alpha)(\gamma - \beta)z}
\end{pmatrix}, \qquad & 
\det A_3(\ba) &= 
\frac{\gamma(\gamma + 1)(z - 1)}{(\gamma - \alpha)(\gamma - \beta)z}. 
\label{eqn:A3}
\end{alignat} 
\end{subequations}
The matrix $A_{13}(\ba) := A(\ba; \be_1 + \be_3) = A_1(\ba + \be_3)A_3(\ba)$ 
also appears frequently.  
For this matrix we have    
\begin{equation} \label{eqn:A13}
A_{13}(\ba)
= \frac{1}{\gamma-\beta} 
\begin{pmatrix}
\gamma & \beta(z-1) \\[2mm]
- \frac{\gamma (\gamma+1)}{(\alpha+1)z} & 
\frac{(\gamma+1)(\gamma - \beta z)}{ (\alpha+1)z}
\end{pmatrix}, 
\qquad 
\det A_{13}(\ba) = \frac{\gamma(\gamma+1)}{(\alpha+1)(\gamma-\beta) z}.   
\end{equation}
Let $(\alpha)_k$ denote the Pochhammer symbol: 
$(\alpha)_0 = 1$ and $(\alpha)_k := \alpha(\alpha+1) \cdots (\alpha + k-1)$ for 
$k \in \bN$. 
\begin{lemma} \label{lem:A(a;p)}
For any $\bp = (p, 0, r) \in \bZ^3$ with $1 \le p \le r$, the matrix $A(\ba; \bp)$ 
can be represented as   
\begin{equation} \label{eqn:A(a;p)}
A (\ba; \bp) 
= \dfrac{1}{(\gamma - \alpha)_{r - p} ( \gamma - \beta)_{r}}   
\begin{pmatrix} 
\dfrac{(\gamma)_{r} \cdot \phi_{11}^{(r - 1)}}{(\alpha + 1)_{p - 1}} & 
\dfrac{(\gamma + 1)_{r - 1}\cdot \phi_{12}^{(r - 1)}}{(\alpha + 1)_{p - 1}} 
\\[4mm]
\dfrac{(\gamma)_{r + 1} \cdot \phi_{21}^{(r - 1)}}{(\alpha + 1)_{p}} & 
\dfrac{(\gamma + 1)_{r} \cdot \phi_{22}^{(r)}}{(\alpha + 1)_{p}}
\end{pmatrix},  
\end{equation} 
where $\phi_{i j}^{(k)} = \phi_{i j}^{(k)} (\ba; \bp)$ is a polynomial 
of $\ba = (\alpha, \beta, \gamma)$ with degree at most $k$ in $(\alpha, \gamma)$, 
not in $\ba$.     
Moreover,   
\begin{equation} \label{eqn:detA(a;p)}
\det A (\ba; \bp) = 
\dfrac{z^{- r}(z - 1)^{r - p}(\gamma)_{r}(\gamma+1)_{r}}
{(\alpha + 1)_{p}(\gamma-\alpha)_{r - p} (\gamma-\beta)_{r}}. 
\end{equation}
\end{lemma}
{\it Proof}. 
We prove the assertion other than formula \eqref{eqn:detA(a;p)} by induction on $r$. 
First, for $r=1$ we have $\bp = \be_1 + \be_3$ and hence $A(\ba ; \bp) = A_{13}(\ba)$.   
Comparing \eqref{eqn:A13} and \eqref{eqn:A(a;p)}, we have     
$\phi_{11}^{(0)} = 1$, $\phi_{12}^{(0)} = \beta (z - 1)$, 
$\phi_{21}^{(0)} = -1/z$ and $\phi_{22}^{(1)} = (\gamma - \beta z)/z$.  
Thus the assertion holds for $r = 1$. 
Next we show that the assertion stays true under the passage $r \mapsto r+1$. 
Let $\bp \mapsto \bar{\bp}$ be the corresponding change in $\bp$. 
Then there are two cases:   
\begin{center}
(i) \, $\bp = (p, 0, r)$ and $\bar{\bp} = (p, 0, r+1)$ with $p \le r$; \qquad 
(ii) \, $\bp = (r, 0, r)$ and $\bar{\bp} = (r+1, 0, r+1)$. 
\end{center} 
\par
In case (i) we have $A(\ba; \bar{\bp}) = A_3(\ba + \bp) \, A(\ba;\bp)$.  
Using formulas \eqref{eqn:A(a;p)} and \eqref{eqn:A3},  
we have recursion relations: 
\begin{alignat*}{2}
\phi_{11}^{(r)} 
&= (\gamma-\alpha-\beta+r-p) \, \phi_{11}^{(r-1)}
-\beta (z-1) \, \phi_{21}^{(r-1)}, 
\quad & 
\phi_{12}^{(r)} 
&= (\gamma-\alpha-\beta+r-p) \, \phi_{12}^{(r-1)}
-\beta (z-1) \, \phi_{22}^{(r)}, 
\\[1mm]
\phi_{21}^{(r)} &
= z^{-1} \{(\alpha+p) \, \phi_{11}^{(r-1)} 
+(\gamma+r) (z-1) \, \phi_{21}^{(r-1)} \}, 
\quad & 
\phi_{22}^{(r+1)} 
&= z^{-1} \{(\alpha + p) \, \phi_{12}^{(r-1)} 
+(\gamma + r) (z-1) \, \phi_{22}^{(r)} \}.   
\end{alignat*}
\par
In case (ii) we have $A(\ba; \bar{\bp}) = A_{13}(\ba + \bp) A(\ba;\bp)$. 
Using formulas \eqref{eqn:A(a;p)} and \eqref{eqn:A13}, 
we have recursion relations:  
\begin{alignat*}{2}
\phi_{11}^{(r)} 
&= (\alpha + r) \, \phi_{11}^{(r-1)} + \beta (z - 1) \, \phi_{21}^{(r-1)}, 
\qquad & 
\phi_{12}^{(r)} 
&= (\alpha + r) \, \phi_{12}^{(r-1)} + \beta (z - 1) \, \phi_{22}^{(r)}, 
\\[1mm]
\phi _{21}^{(r)} 
&= - (\alpha + r) \, \phi_{11}^{(r-1)} 
+ (\gamma + r -\beta z) \, \phi_{21}^{(r-1)}, 
\qquad & 
\phi_{22}^{(r+1)} 
&= - (\alpha + r) \, \phi_{12}^{(r-1)} 
+ (\gamma + r -\beta z) \, \phi_{22}^{(r)}.   
\end{alignat*}
\par
In either case the recursion relations show that the assertion about 
the degree remains true under the change $r \mapsto r+1$, so the 
induction is complete. 
Finally we take the determinant of 
\begin{align*}
A(\ba; \bp) 
&= A_1(\ba +(p-1)\be_1+ r \be_3) \cdots A_1(\ba+ \be_1+ r \be_3 ) \, A_1(\ba+r \be_3 ) \\ 
&\hspace{10mm} \cdot A_3(\ba +(r-1) \be_3) \cdots A_3(\ba + \be_3) \, A_3 (\ba)
\end{align*} 
and use the determinant formulas in \eqref{eqn:A1-A3} to establish formula   
\eqref{eqn:detA(a;p)}. \hfill $\Box$ \par\medskip
Given integral data $\lambda = (p, q, r; a, b; x)$, putting $\ba = (p w + a, q w + b; r w)$ 
and $z = x$ into equation \eqref{eqn:m-3tr} yields 
\begin{equation*} \label{eqn:m-contig2}
\bmf(w+1;\lambda) 
= A(w; \lambda) \bmf(w;\lambda), \qquad 
\bmf(w;\lambda) := {}^t(f (w;\lambda), \, \widetilde{f} (w;\lambda)),  
\end{equation*}
where $A(w; \lambda)$ is a matrix depending on $\lambda$ whose entries are rational 
function of $w$. 
Then the coefficients of the three-term relation 
\eqref{eqn:ttr2} can be represented as  
\begin{equation} \label{eqn:RQA}
R (w;\lambda) = A(w;\lambda)_{11}, \qquad Q(w;\lambda) = A(w;\lambda)_{12}. 
\end{equation}
where $A(w;\lambda)_{ij}$ stands for the $(i, j)$-entry of the matrix $A(w; \lambda)$. 
Lemma \ref{lem:A(a;p)} has the following corollary. 
\begin{corollary} \label{cor:A(w;l)} 
For any piece of integral data $\lambda = (p, 0, r; a, 1/2; x) \in \cI$ we have  
\begin{equation}  \label{eqn:A(w;l)}
A (w; \lambda ) = \dfrac{1}{((r - p) w - a)_{r - p} (r w - 1/2)_{r }}   
\begin{pmatrix}
\dfrac{(r w)_{r} \cdot \phi_{11}^{(r - 1)} (w; \lambda)}{(p w + a + 1)_{p - 1}} & 
\dfrac{(r w + 1)_{r - 1} \cdot \phi_{12}^{(r - 1)}(w; \lambda)}{(p w + a + 1)_{p - 1}} \\[4mm]
\dfrac{(r w)_{r + 1} \cdot \phi_{21}^{(r - 1)}(w; \lambda)}{(p w + a + 1)_{p}} & 
\dfrac{(r w + 1)_{r} \cdot \phi_{22}^{(r)}(w; \lambda)}{(p w + a + 1)_{p}}
\end{pmatrix},  
\end{equation}
where $\phi_{i j}^{(k)} (w; \lambda)$ is a polynomial in $w$ of degree at most $k$ 
depending on $\lambda$.  
Moreover we have  
\begin{equation} \label{eqn:detA(w;l)}
\det A (w; \lambda) = \dfrac{x^{- r}(x - 1)^{r - p } (r w)_{r}(r w + 1)_{r}}
{(p w + a + 1)_{p} ((r - p) w - a)_{r - p} ( r w - 1/2)_{r}}. 
\end{equation}
\end{corollary}
{\it Proof}. 
We have only to substitute $\ba = (p w + a, 1/2, r w)$ and $z = x$ into formulas \eqref{eqn:A(a;p)} 
and \eqref{eqn:detA(a;p)}. 
\hfill $\Box$
\begin{lemma} \label{lem:R(w;l)}
A piece of integral data $\lambda = (p, 0, r; a, 1/2; x) \in \cI$ is a non-elementary solution 
if and only if the polynomial $\phi_{12}^{(r-1)}(w; \lambda)$ in $w$ vanishes identically, 
in which case $\deg \phi_{11}^{(r-1)}(w; \lambda) = r-1$, $\deg \phi_{22}^{(r)}(w; \lambda) = r$ and 
\begin{equation} \label{eqn:phi}
\phi_{11}^{(r-1)}(w; \lambda) \cdot \phi_{22}^{(r)}(w; \lambda) 
= C_1
\prod_{i=1}^{p-1} \Big( w + \ts \frac{i + a}{p} \Big) 
\ds \prod_{j=0}^{r-p-1} \Big(w + \ts \frac{j-a}{r-p} \Big) 
\ds \prod_{j=0}^{r-1} \Big(w + \ts \frac{j- 1/2}{r} \Big),   
\end{equation}
where $C_1 := (r/x)^{r}(x - 1)^{r - p} p^{p-1} (r-p)^{r-p}$. 
In this case, moreover, we have   
\begin{equation} \label{eqn:R(w;l)}
R (w; \lambda ) = (r/x)^{r} (x-1)^{r - p} \cdot 
\dfrac{\prod_{i=0}^{r-1} \big(w + \frac{i}{r} \big)}{\phi_{22}^{(r)}(w; \lambda)}.  
\end{equation}
\end{lemma}
{\it Proof}.  
By Theorem \ref{thm:cfcr}, $\lambda$ is a solution if and only if 
$A(w; \lambda)_{12} = Q(w; \lambda) \equiv 0$, that is, $\phi_{12}^{(r-1)}(w; \lambda) \equiv 0$,  
in which case the matrix $A(w; \lambda)$ is lower triangular.  
Thus taking the determinant of equation \eqref{eqn:A(w;l)} yields 
\begin{equation*}
\det A(w;\lambda) = \frac{1}{\{ ((r - p) w - a)_{r - p} ( r w -1/2)_{r } \}^{2}}
\cdot \dfrac{(r w)_{r} \cdot \phi_{11}^{(r - 1)}(w; \lambda)}{(p w + a + 1)_{p - 1}} \cdot 
\dfrac{(r w + 1)_{r} \cdot \phi_{22}^{(r)}(w; \lambda)}{(p w + a + 1)_{p}}. 
\end{equation*}
Comparing this expression with formula \eqref{eqn:detA(w;l)}, we have 
\begin{equation} \label{eqn:phi-d}
\phi_{11}^{(r-1)}(w; \lambda) \cdot \phi_{22}^{(r)}(w; \lambda) 
= x^{-r}(x - 1)^{r - p}\cdot(p w + a + 1)_{p - 1}((r - p )w - a)_{r - p}(r w - 1/2)_{r},   
\tag{$\ref{eqn:phi}'$}
\end{equation}
which is equivalent to formula \eqref{eqn:phi}. 
From Corollary \ref{cor:A(w;l)} we have $\deg \phi_{11}^{(r-1)}(w; \lambda) \le r-1$ and 
$\deg \phi_{22}^{(r)}(w; \lambda) \le r$, while formula \eqref{eqn:phi} implies    
$\deg \{ \phi_{11}^{(r-1)}(w; \lambda) \cdot \phi_{22}^{(r)}(w; \lambda) \} = 2r -1 = 
(r-1) + r$. 
This forces $\deg \phi_{11}^{(r-1)}(w; \lambda) = r-1$ and 
$\deg \phi_{22}^{(r)}(w; \lambda) = r$.  
Finally, it follows from formulas \eqref{eqn:RQA}, \eqref{eqn:A(w;l)} and \eqref{eqn:phi-d} that   
\begin{equation*}
R(w; \lambda) = 
\dfrac{ (r w)_r \cdot \phi_{11}^{(r-1)}(w; \lambda) }{ ((r - p) w - a)_{r - p} \cdot 
(r w - 1/2)_{r } \cdot (p w + a + 1)_{p - 1}} = 
 x^{-r} (x-1)^{r - p} \dfrac{(r w)_{r}}{\phi_{22}^{(r)}(w; \lambda)},       
\end{equation*}
which is equivalent to formula \eqref{eqn:R(w;l)}. 
This proves Lemma \ref{lem:R(w;l)}. \hfill $\Box$ \par\medskip 
Lemma \ref{lem:R(w;l)} will be important in Proposition \ref{prop:p-p} as an origin  
of the GPF for a non-elementary solution in $\cI$. 
\subsection{Leading Asymptotics of Contiguous Matrices}  \label{ss:p-p}
The leading asymptotics of the contiguous matrix $A(w; \lambda)$ as $w \to \infty$ 
contains some useful information.  
\begin{lemma} \label{lem:p-p} 
For any piece of integral data $\lambda = (p, 0, r; a, 1/2; x) \in \cI$, the matrix   
$A(w; \lambda)$ can be expanded as 
\begin{equation} \label{eqn:p-p}
A(w;\lambda) = 
\begin{pmatrix}
1 & 0 \; \\[1mm]
* & 
\frac{r^r (x-1)^{r-p}}{p^p(r-p)^{r-p} x^r} 
\end{pmatrix}  
+ 
\begin{pmatrix}
* & C_2 \cdot Y(x; \bp)  \\[1mm]
* & * 
\end{pmatrix} 
w^{-1} +O \left(w^{-2} \right) \quad \mbox{as} \quad w \to \infty,    
\end{equation}
where $C_2 := p^{1-p} (r-p)^{p-r} (2r)^{-1} x^{1-r} (1-x) (r- p x)^{-1}$ and 
$Y(z; \bp)$ is a polynomial of $z$ defined by 
\begin{equation} \label{eqn:Y} 
Y(z; \bp) := p^p (r-p)^{r-p} z^r - r^r (1-z)^{r-p},         
\end{equation}
which depends only on the principal part $\bp = (p, 0, r)$ of the data $\lambda$.  
\end{lemma}
{\it Proof}.  It follows from formula \eqref{eqn:mp} with lattice path  
$\bp = (p, 0, r) = (r-p) \be_3 + p (\be_1 + \be_3)$ that    
\[
A(w; \lambda) = \prod_{i=0}^{r-p-1} A_3(p w+a+p, \, 1/2, \, r w+p+i) 
\prod_{j=0}^{p-1} A_{13}(p w+a+j, \, 1/2, \, r w+j) \quad 
\mbox{at} \quad z = x,    
\]
where the matrix products extend from right to left as 
index increases.  
From formulas \eqref{eqn:A3} and \eqref{eqn:A13} we have 
\begin{align*} 
A_3(p w+a+p, \, 1/2, \, r w+p+i) 
&=A_3(0) + A_3^{(i)}(1) \, w^{-1} +  O \left( w^{-2} \right), \\
A_{13}(p w+a+j, \, 1/2, \, r w + j) 
&= A_{13}(0) + A_{13}^{(j)}(1) \, w^{-1} +  O \left( w^{-2} \right),      
\end{align*}
where $A_3(0)$ and $A_{13}(0)$ are lower triangular matrices of the forms 
$$
A_{3}(0) =   
\begin{pmatrix}
1 & 0 \\
* & \frac{r (x-1)}{(r-p) x}
\end{pmatrix}, 
\qquad 
A_{13}(0) =  
\begin{pmatrix}
1 & 0 \\
* & \frac{r}{p x}
\end{pmatrix}, 
$$
which are independent of $i$ and $j$, while $A_3^{(i)}(1)$ and $A_{13}^{(j)}(1)$ are 
matrices of the forms      
$$
A_3^{(i)}(1) = 
\begin{pmatrix}
* & -\frac{p(x-1)}{2(r-p)r} \\
* & *
\end{pmatrix},   
\qquad 
A_{13}^{(j)}(1) =  
\begin{pmatrix}
* & \frac{x-1}{2 r} \\
* & *
\end{pmatrix},  
$$
which may depend on $i$ and $j$ respectively.   
So the leading term in expansion \eqref{eqn:p-p} is the lower triangular matrix  
\[
A_3(0)^{r-p} \cdot A_{13}(0)^p = 
\begin{pmatrix} 
1 & 0 \\[1mm] 
* & \frac{r^r (x-1)^{r-p} }{ p^p (r-p)^{r-p} x^r} 
\end{pmatrix}, 
\]
whereas the coefficient of the second leading term is the matrix sum 
\[
\sum_{i=0}^{r-p-1} 
A_3(0)^i \cdot A_3^{(i)}(1) \cdot A_3(0)^{r-p-1-i} \cdot A_{13}(0)^p 
+ 
\sum_{j=0}^{p-1} A_3(0)^{r-p} \cdot A_{13}(0)^j \cdot 
A_{13}^{(j)}(1) \cdot A_{13}(0)^{p-1-j},    
\]
whose $(1, 2)$-entry is evaluated as 
\begin{gather*}
\sum_{i=0}^{r-p-1} 1^i \cdot \left\{ - \ts \frac{p(x-1)}{2(r-p)r} \right\} \cdot
\left\{ \ts \frac{r(x-1)}{(r-p) x} \right\}^{r-p-1-i} \cdot 
\left( \ts \frac{r}{p x} \right)^p + 
\sum_{j=0}^{p-1} 1^{r-p} \cdot 1^j \cdot \left( \ts \frac{x-1}{2 r} \right) 
\cdot \left( \ts \frac{r}{ p x} \right)^{p-1-j} \\[1mm]
= \frac{(1-x) \{  p^p(r-p)^{r-p} x^r - r^r (x-1)^{r-p} \} }{2 r \cdot 
p^{p-1} (r-p)^{r-p} x^{r-1}(r-p x)} = C_2 \cdot Y(x; \bp). 
\end{gather*}
Therefore the expansion formula \eqref{eqn:p-p} is established. 
\hfill $\Box$ \par\medskip
Lemma \ref{lem:p-p} leads to an algebraic equation for $x$, which is 
a preliminary to the one \eqref{eqn:alg-eq} in Theorem \ref{thm:a-x}.   
\begin{corollary} \label{cor:alg-eq}
If $\lambda = (p, 0, r; a, 1/2; x) \in \cI$ is a non-elementary solution,  
then $x$ is the unique root in the interval $(0, \, 1)$ of the algebraic equation 
$Y(z; \bp) = 0$, where $Y(z; \bp)$ is defined in formula \eqref{eqn:Y}.     
\end{corollary} 
{\it Proof}. 
By Lemma \ref{lem:R(w;l)} the matrix $A(w; \lambda)$ is lower triangular, 
so formula \eqref{eqn:p-p} in Lemma \ref{lem:p-p} forces $Y(x; \bp) = 0$, 
hence $x$ is a root of the equation $Y(z; \bp) = 0$. 
Since $Y(0; \bp) = -r^r < 0 < Y(1; \bp) = p^p (r-p)^{r-p}$ and 
$$
Y'(z; \bp) = r(r-p) \{ p^p (r - p)^{r-p-1} z^{r-1} + r^{r-1} (1-z)^{r-p-1} \} > 0 \quad 
\mbox{for} \quad  0 < z < 1, 
$$
the equation $Y(z; \bp) = 0$ has a unique root $x$ in the interval $(0, \, 1)$.  
\hfill $\Box$ \par\medskip
We are now able to determine the general form of GPF every 
non-elementary solution in $\cI$ must take.   
To state it we denote by $[k]$ the set $\{ 0, 1, \dots, k-1 \}$ 
for a given positive integer $k$.  
\begin{proposition} \label{prop:p-p} 
For any non-elementary solution $\lambda = (p, 0, r; a, 1/2; x) \in \cI$ 
there exist subsets $I_p \subset [p] \setminus \{ 0 \}$, 
$J_p \subset [r-p]$ and $J_q \subset [r]$ with total cardinality 
$|I_p| + |J_p| + |J_q| = r$ such that $\lambda$   
admits a GPF of the form
\begin{equation} \label{eqn:gpf1a}
f(w; \lambda) 
= C \cdot \dfrac{\prod_{i=0}^{r-1} \vG(w+ \frac{i}{r})}{\prod_{i=1}^r \vG(w + v_i)},   
\tag{$\ref{eqn:gpf1}$}
\end{equation} 
where $C$ is a nonzero constant and $v_1, \dots, v_r$ are determined by 
the equation 
\begin{equation} \label{eqn:division} 
\prod_{i=1}^r (w+v_i) =  
\prod_{i \in I_p} \Big(w+ \ts\frac{i+a}{p} \Big) 
\ds \prod_{j \in J_p} \Big( w + \ts\frac{j-a}{r-p} \Big) 
\ds \prod_{j \in J_q} \Big( w + \ts\frac{j-1/2}{r} \Big).    
\end{equation}
Moreover the numbers $v_1, \dots, v_r$ sum to $v_1 + \cdots + v_r = (r-1)/2$.   
\end{proposition}
{\it Proof}.  
By Lemma \ref{lem:R(w;l)} we have $\deg \phi_{22}^{(r)}(w; \lambda) = r$,   
so we can write $\phi_{22}^{(r)}(w; \lambda) = c_1 (w+v_1) \cdots (w+v_r)$ with 
a nonzero constant $c_1$ and some numbers $v_1, \dots, v_r$.  
By equation \eqref{eqn:phi} there exist subsets $I_p \subset [p] \setminus \{ 0 \}$, 
$J_p \subset [r-p]$ and $J_q \subset [r]$ with $|I_p| + |J_p| + |J_q| = r$ such that 
equation \eqref{eqn:division} is satisfied.   
From formula \eqref{eqn:R(w;l)} we have   
$$
\dfrac{f(w+1; \lambda)}{f(w; \lambda)} = R(w; \lambda) = d \cdot 
\frac{\prod_{i=0}^{r-1} \left( w+ \frac{i}{r}\right)}{\prod_{i=1}^r (w+v_i)},  
$$
for some constant $d$. 
Since $R(w; \lambda) = A(w; \lambda)_{11} \to 1$ as $w \to \infty$ in 
formula \eqref{eqn:p-p}, we have $d = 1$. 
Thus GPF \eqref{eqn:gpf1} follows from Theorem \ref{thm:lift}, and  
the summation $v_1 + \cdots + v_r = (r-1)/2$ comes from formula \eqref{eqn:sum2}.   
\hfill $\Box$ \par\medskip
Proposition \ref{prop:p-p} establishes GPF \eqref{eqn:gpf1} and the first half 
of summation formula \eqref{eqn:sum} in Theorem \ref{thm:gpf}. 
\subsection{Truncated Hypergeometric Products} \label{ss:trunc}
In view of the significant roles played by the rational functions $Q(w; \lambda)$ and $R(w; \lambda)$, 
it is important to find effective formulas for them. 
This amounts to having such formulas for the polynomials $\phi_{12}^{(r-1)}(w; \lambda)$ 
and $\phi_{22}^{(r)}(w; \lambda)$, since the former are connected to the latter 
by the relations \eqref{eqn:RQA}, \eqref{eqn:A(w;l)} and \eqref{eqn:R(w;l)}.  
The aim of this subsection is to express those polynomials in terms of 
truncations of products of two hypergeometric series; see Lemma \ref{lem:Phi-P}. 
\par
Given a power series $\phi (z) = \sum_{j=0}^{\infty}c_{k} \, z^{j}$, its {\sl truncation} 
at degree $k$ is the polynomial $\left<\phi(z)\right>_{k} := \sum_{j=0}^{k} c_{k} \, z^{j}$. 
In what follows truncation is always taken with respect to the variable $z$. 
First we recall from \cite[Lemma 10.1]{Iwasaki1} the following lemma, which is 
originally due to Ebisu \cite[Proposition 3.4, Theorem 3.7, Remark 3.11]{Ebisu1}.     
\begin{lemma} \label{lem:ebi} 
For any integer vector $\bp = (p, q, r) \in \bZ^3$ with $0 \le q \le p \le r-1$ we have  
\begin{subequations} \label{eqn:q(a)}
\begin{align} 
q(\ba ;\bp; z) &= z^{1-r} (z - 1) \cdot C ( \ba; \bp) \cdot 
\langle \hgF(\ba^*; z) \cdot \hgF(\bv - \ba^* - \bp^*; z) \rangle_{r-q-1}, 
\label{eqn:q(a;p)}  \\
q (\ba ;\bp+ \1; z) &= z^{-r} (z - 1) \cdot C (\ba; \bp + \1) 
\cdot \langle \hgF(\ba^*; z)\cdot \hgF(\1 - \ba^* - \bp^*; z) \rangle_{r-q-1},  
\label{eqn:q(a;p+1)} 
\end{align}
\end{subequations}
where $\ba := (\alpha, \beta, \gamma)$, 
$\ba^* := (\gamma - \alpha, \gamma - \beta, \gamma)$, 
$\bp^* := (r-p, r-q, r)$, $\bv := (1, 1, 2)$, $\1 := (1, 1, 1)$ and 
$$
C(\ba ; \bp) := 
(-1)^{r-p-q} \cdot 
\frac{ (\gamma)_{r-1}(\gamma + 1)_{r-1}}{(\alpha + 1)_{p-1} (\beta+1)_{q-1} 
(\gamma - \alpha)_{r-p}(\gamma - \beta)_{r-q}}, \qquad 
(\alpha+1)_{-1} := \alpha^{-1}.  
$$
\end{lemma}
\par
Given a piece of integral data $\lambda = (p, 0, r; a, 1/2; x) \in \cI$, 
we consider truncated hypergeometric products:  
\begin{subequations} \label{eqn:PhiP}
\begin{align} 
\Phi (w ; \lambda) &: = (r w)_{r-1} \cdot 
\langle \hgF(\bal^*(w); z) \cdot \hgF(\bv - \bal^*(w+1) ; z) \rangle_{r - 1} \big|_{z = x},  
\label{eqn:Phi} \\  
P(w; \lambda) &:= (r w)_{r} \cdot 
\langle \hgF(\bal^*(w) ; z) \cdot \hgF(\1 - \bal^*(w + 1); z) \rangle_{r - 1} \big|_{z=x},  
\label{eqn:P}
\end{align}
\end{subequations} 
where $\bal^*(w) := ( (r - p)w-a, \, r w-1/2; \, r w)$. 
Then we have the following lemma.      
\begin{lemma} \label{lem:Phi-P}
For any piece of integral data $\lambda = (p, 0, r; a, 1/2; x) \in \cI$ we have  
\begin{subequations} \label{eqn:phi122}
\begin{align}
\phi_{12}^{(r - 1)}(w; \lambda) &= (-1)^{r-p} \cdot 2^{-1} \cdot 
x^{1-r} \cdot (x - 1) \cdot \Phi (w; \lambda),  
\label{eqn:phi12} \\
\phi_{22}^{(r)}(w; \lambda) &= (-1)^{r-p-1} x^{-r} (x-1) \cdot P(w;\lambda).  
\label{eqn:phi22}
\end{align}
\end{subequations}
\end{lemma}
{\it Proof}. 
Substituting $\bp = (p, 0, r)$, $\ba = \bal(w)$, $\ba^* = \bal^*(w)$ and $z = x$ into 
formula \eqref{eqn:q(a;p)}, we have    
\begin{align*}
A(w; \lambda)_{12} &= q(\bal(w); \bp; x) \\  
&= x^{1-r} \cdot (x-1) \cdot C(\bal(w);\bp) 
\langle \hgF(\bal^*(w); z) \cdot 
\hgF(\bv-\bal^*(w) - \bp^*; z) \rangle_{r-1} \big|_{z = x} 
\\[1mm]
&= x^{1-r} \cdot (x-1) \cdot 
C(\bal(w);\bp) 
\langle \hgF(\bal^*(w); z) \cdot 
\hgF(\bv - \bal^*(w+1); z) \rangle_{r-1} \big|_{z=x} 
\\[1mm]
&= x^{1-r} (x-1) \cdot 
C(\bal(w);\bp) \cdot 
\Phi (w ; \lambda)/(r w)_{r-1},  
\end{align*} 
where $\bal^*(w) + \bp^* = \bal^*(w+1)$ is used in the third equality.   
On the other hand, formula \eqref{eqn:A(w;l)} implies  
$$
A(w; \lambda)_{12} = \frac{(r w+1)_{r-1} \cdot 
\phi_{12}^{(r-1)}(w;\lambda)}{(p w+a+1)_{p-1}((r-p)w-a)_{r-p}(r w-b)_{r}}. 
$$
Comparing these two expressions for $A(w; \lambda)_{12}$, we obtain formula \eqref{eqn:phi12}. 
\par
Similarly, substituting $\bp = (p, 0, r)$, $\ba = \bal(w)$, $\ba^* = \bal^*(w)$ 
and $z = x$ into formula \eqref{eqn:q(a;p+1)}, we have  
\begin{align*}
A(w; \lambda)_{22} &= q(\bal(w); \bp + \1; x) \\ 
&= x^{-r} \cdot (x-1) \cdot 
C(\bal(w); \bp + \1) \cdot \langle \hgF(\bal^*(w); z) 
\cdot \hgF(\1-\bal^*(w)-\bp^*; z) \rangle_{r-1} \big|_{z=x} 
\\[1mm] 
&= x^{-r} \cdot (x-1) \cdot 
C(\bal(w);\bp + \1) \cdot \langle \hgF(\bal^*(w); z) \cdot 
\hgF(\1 - \bal^*(w + 1); 
z \rangle_{r-1} \big|_{z=x} 
\\[1mm]  
&= x^{-r} (x-1) \cdot C(\bal(w);\bp + \1) \cdot P(w;\lambda)/(r w)_{r}.    
\end{align*}
On the other hand, formula \eqref{eqn:A(w;l)} implies  
$$
A(w;\lambda)_{22} = 
\frac{(r w + 1)_{r} \cdot \phi_{22}^{(r)}(w; \lambda)}{((r-p)w-a)_{r-p} 
(r w-1/2)_{r} (p w +a+1)_{p}}.  
$$
Comparing these two expressions for $A(w; \lambda)_{22}$, 
we obtain formula \eqref{eqn:phi22}. 
\hfill $\Box$ \par\medskip
In terms of truncated hypergeometric products, Theorem \ref{thm:cfcr} 
and Proposition \ref{prop:p-p} are summed up as follows. 
\begin{proposition} \label{prop:Phi-P} 
A piece of integral data $\lambda = (p, 0, r; a, 1/2; x) \in \cI$ is a solution 
if and only if the polynomial $\Phi(w; \lambda)$ vanishes identically, 
in which case we have $P(w; \lambda) =  c (w + v_1) \cdots (w + v_r)$, 
where $c$ is a nonzero constant and $v_1, \dots, v_r$ are the 
numbers appearing in GPF \eqref{eqn:gpf1a} in Propisition $\ref{prop:p-p}$.    
\end{proposition}
{\it Proof}. 
Note that $\lambda$ is a solution if and only if $\phi_{12}^{(r-1)}(w; \lambda) \equiv 0$ 
by Lemma \ref{lem:R(w;l)},   
if and only if $\Phi(w; \lambda) \equiv 0$ by formula \eqref{eqn:phi12}. 
In this case, formula \eqref{eqn:phi22} implies that $P(w; \lambda)$ 
is a nonzero constant multiple of $(w + v_1) \cdots (w + v_r)$, since so is 
$\phi_{22}^{(r)}(w; \lambda)$ as stated in the proof of Proposition \ref{prop:p-p}.   
\hfill $\Box$ 
\subsection{Terminating Hypergeometric Sums} \label{ss:thgs}
If the truncated hypergeometric products $\Phi(w; \lambda)$ and 
$P(w; \lambda)$ in \eqref{eqn:PhiP} are evaluated at the points 
\begin{equation} \label{eqn:xi-eta}
\xi_j := - \frac{j-a}{r-p}, \quad j \in [r-p]; \qquad 
\eta_j := - \frac{j-1/2}{r}, \quad j \in [r],  
\end{equation}
then the two hypergeometric series therein reduce to finite sums,  
rendering it unnecessary to take truncations. 
As is already observed in \cite[\S 11]{Iwasaki1} for a different region 
of data $\lambda$, the (renormalized) hypergeometric polynomials  
\begin{equation} \label{eqn:cF}
\cF_{k}(\beta; \gamma; z) :=
\sum_{j=0}^{k}(-1)^{j}\binom{k}{j}(\beta)_{j}(\gamma + j)_{k-j} \, z^{j}, \qquad 
k \in \bZ_{\ge0},  
\end{equation} 
comes up in this context. 
Note that $\cF_k(\beta; \gamma; z) = (\gamma)_{k} \cdot 
\hgF(-k,\beta; \gamma; z)$, unless $\gamma$ is zero or a negative integer.  
\begin{lemma} \label{lem:sol-p}
Let $\lambda = (p, 0, r; a, 1/2; x) \in \cI$ be a piece of integral data 
and $j \in [r-p]$. 
If $r \xi_j \not \in \bZ$, then   
\begin{subequations} \label{eqn:sol-p}
\begin{align} 
\Phi (\xi_j; \lambda) 
&= (-1)^{r-p-1-j} \cdot (r \xi_j + j)_{p} \cdot \cF_j(r \xi_j -1/2; r \xi_j; x) 
\cdot \cF_{r-p-1-j} ( b_j -1/2; b_j; x ),  \label{eqn:sol-p1} \\ 
P ( \xi_j; \lambda) &= (-1)^{r-p-1-j} \cdot ( r \xi_j + j)_{p+1} \cdot 
\cF(r \xi_j -1/2; r \xi_j; x) \cdot \cF_{r-p-1-j} ( b_j-1/2; b_j-1; x ),    
 \label{eqn:sol-p2}
\end{align}
\end{subequations}
where $\xi_j$ is defined in formula \eqref{eqn:xi-eta} and $b_j := 2 - r(\xi_j + 1)$. 
\end{lemma}
{\it Proof}.   
Substituting $w = \xi_j$ into definitions \eqref{eqn:PhiP} and observing that 
\begin{align*}
\bal^*(\xi_j) &= ( -j, \, r \xi_j -1/2, \, r \xi_j),  \\ 
\bv-\bal^*(\xi_j +1) &= (-(r-p-1-j), \, 3/2-r( \xi_j + 1), \, 2- r( \xi_j +1) ), \\
\1-\bal(\xi_j +1) &= ( -(r-p-1-j), \, 3/2 -r( \xi_j +1), \, 1 -r( \xi_j + 1)), \\
(r \xi_j)_{r-1} &= (-1)^{r-p-1-j} \cdot (r \xi_j + j)_p \cdot (r \xi_j)_j \cdot (2-r( \xi_j+1))_{r-p-1-j}, \\
(r \xi_j)_r &= (-1)^{r-p-1-j} \cdot (r \xi_j + j)_{p+1} \cdot (r \xi_j)_j \cdot (1-r(\xi_j +1))_{r-p-1-j},  
\end{align*}
with $r \xi_j \not \in \bZ$ and $b_j = 2- r( \xi_j + 1) \not \in \bZ$, 
we obtain formulas \eqref{eqn:sol-p} from definition \eqref{eqn:cF}.  
\hfill $\Box$ 
\begin{lemma} \label{lem:sol-q}
For any piece of integral data $\lambda = (p, 0, r; a, 1/2; x) \in \cI$ and any $j \in [r]$ 
we have 
\begin{subequations} \label{eqn:sol-q} 
\begin{align}
\Phi(\eta_j; \lambda) &= (-1)^{r-1-j} \cdot 
\cF_j( c_j; \, 1/2 - j; \, x) \cdot \cF_{r-1-j}(d_j; \, 3/2 + j-r; \, x), 
\label{eqn:sol-q1} \\ 
2 \, P(\eta_j; \lambda) &= (-1)^{r-1-j} \cdot 
\cF_j( c_j; \, 1/2 - j; \, x) \cdot \cF_{r-1-j}(d_j; \, 1/2 + j-r; \, x), 
\label{eqn:sol-q2}
\end{align}
\end{subequations}
where $\eta_j$ is defined in formula \eqref{eqn:xi-eta}, 
$c_j := (r-p) \eta_j - a$ and $d_j := 1 - (r-p)  (\eta_j + 1) + a$.    
\end{lemma}
{\it Proof}. 
Substituting $w = \eta_j$ in definitions \eqref{eqn:PhiP} and observing that 
\begin{align*}
\bal^*(\eta_j) &= ( c_j, \, - j, \, r \eta_j),  \\ 
\bv-\bal^*(\eta_j +1) &= (d_j, \, -(r-1-j), \, 2-r(\eta_j +1) ), \\
\1-\bal^*(\eta_j +1) &= ( d_j, \, -(r-1-j), \, 1 -r(\eta_i + 1)), \\
(r \eta_j)_{r-1} &= (-1)^{r-1-j} \cdot  (r \eta_j)_j \cdot (2-r( \eta_j +1))_{r-1-j}, \\
(r \eta_j)_r &= (-1)^{r-1-j} \cdot (r \eta_j + j) \cdot (r \eta_j)_j \cdot (1-r( \eta_j +1))_{r-1-j},  
\end{align*}
with $r \eta_j + j = 1/2$, $r \eta_j = 1/2 - j \not \in \bZ$, 
$2 - r(\eta_j + 1) = 3/2 + j - r \not \in \bZ$ and $1 - r(\eta_j + 1) = 1/2 + j - r \not \in \bZ$,    
we obtain formulas \eqref{eqn:sol-q} from definition \eqref{eqn:cF}. 
\hfill $\Box$ \par\medskip 
In connection with the evaluations mentioned above, we recall from \cite[Lemma 11.2]{Iwasaki1} 
the following lemma.   
\begin{lemma} \label{lem:thgs}
Fix $k \in \bN$, $\beta, \gamma \in \bC$, 
$x \in \bC \setminus \{0,1\}$ and let $z$ be an indeterminate variable.  
\begin{enumerate} 
\setlength{\itemsep}{-1pt}
\item $\cF_{k}(\beta; \gamma; z)$ vanishes identically in $\bC[z]$ 
if and only if $\beta$, $\gamma \in \bZ$ and $0 \le -\beta \le -\gamma \le k -1$. 
\item If $\mathcal{F}_{k}(\beta; \gamma; x) = \cF_{k}(\beta; \gamma-1; x) = 0$ then  
$\cF_{k}(\beta; \gamma; z)$ vanishes identically in $\bC [z]$. 
\item If $\cF_k(\beta; \gamma; x) = \cF_{k-1}(\beta + 1; \gamma + 1; x) = 0$ 
then $\cF_k(\beta; \gamma; z)$ vanishes identically in $\bC[z]$.  
\end{enumerate} 
\end{lemma}
\par
Assertion (1) of Lemma \ref{lem:thgs} provides the condition for the 
vanishing of all coefficients of the polynomial $u(z) := \cF_k(\beta; \gamma; z)$,   
which is a solution to a hypergeometric differential equation.   
The assumptions in (2) and (3) are equivalent to the initial condition 
$u(x) = u'(x) = 0$ at the regular point $x$, which forces $u(z) \equiv 0$.  
\par
The following, seemingly technical, proposition will play an important 
role in proving Lemma \ref{lem:dich}, one of the key lemmas to establish 
some parts of Theorems \ref{thm:a-x} and \ref{thm:gpf}.     
\begin{proposition} \label{prop:avoid} 
In Proposition $\ref{prop:p-p}$. if $r a/(r-p) \not \in \bZ$ then 
$0 \not\in J_p$, while we always have $0 \not \in J_q$. 
\end{proposition}
{\it Proof}. 
By Proposition \ref{prop:Phi-P} we have $\Phi(w; \lambda) \equiv 0$. 
Since $r \xi_0 = r a/(r-p) \not \in \bZ$, 
equations \eqref{eqn:sol-p} with $j = 0$ yield 
\begin{align*}
\Phi(\xi_0; \lambda) &= (-1)^{r-p-1} \cdot (r \xi_0)_p \cdot \cF_{r-p-1}(b_0-1/2; b_0; x) = 0, \\
P(\xi_0; \lambda) &= (-1)^{r-p-1} \cdot (r \xi_0)_{p+1} \cdot \cF_{r-p-1}(b_0-1/2; b_0-1; x),   
\end{align*}
where $(r \xi_0)_p \neq 0$ and $(r \xi_0)_{p+1} \neq 0$. 
Moreover, If $0 \in J_p$, then Proposition \ref{prop:Phi-P} implies $P(\xi_0; \lambda) = 0$, 
and hence $\cF_{r-p-1}(b_0 -1/2; b_0; x) = \cF_{r-p-1}(b_0-1/2; b_0 - 1; x) = 0$.  
By assertion (2) of Lemma \ref{lem:thgs} we have $\cF_{r-p-1}(b_0-1/2; b_0; z) \equiv 0$. 
However, by assertion (1) of the same lemma we have $\cF_{r-p-1}(b_0-1/2; b_0; z) \not \equiv 0$, 
since $b_0 = 2 - r -r a/(r-a) \not \in \bZ$.   
This contradiction shows that if $r a/(r-p) \not \in \bZ$ then $0 \not \in J_p$.  
\par
In a similar mennar, equations \eqref{eqn:sol-q} with $j = 0$ yield 
\begin{align*}
\Phi(\eta_0; \lambda) &= (-1)^{r-1} \cdot \cF_{r-1}(d_0; \, 3/2 -r; \, x) = 0,  \\ 
2 \, P(\eta_0; \lambda) &= (-1)^{r-1} \cdot \cF_{r-1}(d_0; \, 1/2-r; \, x).  
\end{align*}
If $0 \in J_q$ then Proposition \ref{prop:Phi-P} implies $P(\eta_0; \lambda) = 0$ and hence 
$\cF_{r-1}(d_0; \, 3/2 -r; \, x) =\cF_{r-1}(d_0; \, 1/2-r; \, x) = 0$. 
By assertion (2) of Lemma \ref{lem:thgs} we have $\cF_{r-1}(d_0; 3/2 -r; z) \equiv 0$. 
However, by assertion (1) of the same lemma we have $\cF_{r-1}(d_0; 3/2 -r; z) \not \equiv 0$, 
since $3/2-r$ is not an integer. 
This contradiction shows $0 \not \in J_q$. \hfill $\Box$ 
\section{Duality and Reciprocity} \label{sec:d-r}  
In addition to $\hgF(\ba; z)$ the Gauss hypergeometric equation with parameters 
$\ba = (\alpha, \beta, \gamma)$ has solutions 
\begin{align*}
{}_2G_1(\ba; z) 
&:= z^{1-\gamma} (1-z)^{\gamma-\alpha-\beta} \, 
\hgF(1-\alpha, 1-\beta; 2 - \gamma; z),  \\
{}_2H_1(\ba; z) 
&:= z^{1-\gamma}(1-z)^{\gamma-\alpha-\beta} \,  
\hgF(1-\alpha, 1-\beta; \gamma-\alpha-\beta+1; 1-z).    
\end{align*}
Given $\lambda = (p, q, r; a, b ;x) $, in the same manner as $f(w; \lambda)$ is 
derived from $\hgF(\ba; z)$, two additional functions $g(w; \lambda)$ and $h(w; \lambda)$ 
are defined by substituting $\ba = (p w + a, q w + b, rw)$ and $z = x$ into $\hgG(\ba; z)$ 
and $\hgH(\ba; z)$ respectively.  
A straightforward calculation shows that $g(w; \lambda)$ and $h(w; \lambda)$ can be 
represented as 
\begin{subequations} \label{eqn:gh(w;l)}
\begin{alignat}{2}
g (w ; \lambda) 
&= x^{1- rw} (1-x)^{(r-p-q) w -a-b} \, f (w'+1; \lambda'), \qquad 
& w' &:= 2/r -1-w,  \label{eqn:g(w;l)} 
\\[1mm] 
h(w; \lambda)
&= x^{1 - r w} (1-x)^{(r-p-q) w-a-b} \, f(w+c; \check{\lambda}), \qquad 
& c &:= \frac{1-a-b}{r-p-q},   \label{eqn:h(w;l)}
\end{alignat}
\end{subequations}
where $\lambda'$ and $\check{\lambda}$ are the {\sl dual} and {\sl reciplocal} 
to $\lambda$ defined by formulas \eqref{eqn:duality} and \eqref{eqn:recip} respectively.    
\par
Problems $\rI$ and $\rII$ make sense not only for 
$f(w; \lambda)$ but also for $g(w; \lambda)$ and $h(w; \lambda)$. 
Consider those integral solutions to Problem $\rII$ which come from contiguous relations.  
If $\lambda$ is such a solution relative to $f(w; \lambda)$, then $\lambda$ is also 
a solutoin of the same sort relative to $g(w; \lambda)$ and $h(w; \lambda)$; see 
Lemma \cite[Lemma 3.1]{Iwasaki2}.   
This fact and relations \eqref{eqn:gh(w;l)} bring two symmetries to Problem $\rII$, 
which further pass to Problem $\rI$, if solutions to the former problem lift to the latter one.  
This is just the case that occurs in the region $\cI$; see Theorems \ref{thm:lift} 
and \ref{thm:cfcr}. 
\subsection{Duality} \label{ss:dual}
The aim of this subsection is to show that duality $\lambda \mapsto \lambda'$ 
transforms every non-elementary solution in $\cI$ into another in $\cI$,  
and to determine the general form of the dual GPF.   
The result is stated in Proposition \ref{prop:GPF-d}.  
\par
Formula \eqref{eqn:g(w;l)} and the one with $w$ replaced by $w+1$, 
and hence with $w'$ replaced by $w'-1$, yield  
\begin{equation} \label{eqn:ratio1}
\frac{g(w+1; \lambda)}{g(w; \lambda)} = \frac{(1-x)^{r-p-q}}{x^r} \cdot \frac{f(w'; \lambda')}{f(w' + 1; \lambda')}. 
\end{equation}
If $\lambda$ is an integral solution to Problem I\!I comming from contiguous relations, 
then by \cite[Lemma 3.1]{Iwasaki2} we have   
\begin{equation} \label{eqn:ratio2}
\frac{g(w+1; \lambda)}{g(w; \lambda)} = \Psi_g(w; \lambda) \, \frac{f(w+1; \lambda)}{f(w; \lambda)}, 
\end{equation}
where the explicit formula for $\Psi_g(w; \lambda)$ is given in \cite[formula (44)]{Iwasaki2}. 
When $\lambda$ lies in $\cI$, substituting $q = 0$ and $b = 1/2$ into this formula and 
taking the congruence $r - p \equiv 0 \bmod 2$ in \eqref{eqn:integral} into account, 
we have   
\begin{align}
\Psi_{g} (w; \lambda) 
&:= (-1)^{r-p} \dfrac{(p w+a)_p ((r-p)w-a)_{r-p} (r w-1/2)_r}{(r w-1)_r (r w)_r} 
\nonumber \\[2mm]
&= \frac{p^p (r-p)^{r-p}}{r^r} \cdot  
\frac{ \prod_{i=0}^{p-1} ( w + \frac{i+a}{p} ) 
\prod_{j=0}^{r-p-1} ( w+\frac{j-a}{r-p} ) 
\prod_{j=0}^{r-1} (w+\frac{j-1/2}{r} ) }{ \prod_{i=0}^{r-1} ( w + \frac{i-1}{r} ) 
\prod_{i=0}^{r-1} ( w + \frac{i}{r} )}.   \label{eqn:Psi-g}  
\end{align} 
Combining formula \eqref{eqn:ratio1} with \eqref{eqn:ratio2} and replacing $w$ with $w'$, 
we have 
\begin{equation} \label{eqn:ratio3}
\frac{f(w + 1; \lambda')}{f(w; \lambda')} 
= \frac{(1-x)^{r-p}}{x^r \, \Phi_g(w'; \lambda)} \cdot \frac{f(w'; \lambda)}{f(w'+1; \lambda)}, 
\qquad w' := \frac{2}{r} -1 - w. 
\end{equation}
\par
The following proposition establishes the dual GPF \eqref{eqn:gpf2} as well as 
the dual part of summation formula \eqref{eqn:sum} in Theorem \ref{thm:gpf}, thereby 
completing the proofs of Theorems \ref{thm:gpf}.(1) and \ref{thm:a-x}.(1).  
\begin{proposition} \label{prop:GPF-d}
Let $\lambda = (p, 0, r; a , 1/2; x ) \in \cI$ be a non-elementary solution 
with GPF \eqref{eqn:gpf1}.  
Equation \eqref{eqn:division} in Proposition $\ref{prop:p-p}$ allows us 
to introduce the numbers $v_1^*, \dots, v_r^*$ complementary to $v_1, \dots, v_r$ by   
\begin{equation} \label{eqn:division-ast} 
\prod_{i=1}^r (w+v_i) \prod_{i=1}^r (w+v_i^*) =  
\prod_{i=0}^{p-1}\Big(w+ \ts\frac{i+a}{p} \Big) 
\ds \prod_{j=0}^{r-p-1} \Big( w + \ts\frac{j-a}{r-p} \Big) 
\ds \prod_{j=0}^{r-1} \Big( w + \ts\frac{j-1/2}{r} \Big).   
\tag{$\ref{eqn:v*}$}  
\end{equation}
Then the dual $\lambda' \in \cI$ to the solution $\lambda$ is also a 
non-elementary solution having GPF
\begin{equation} \label{eqn:gpf2a}
f(w; \lambda') 
= C' \, \dfrac{\prod_{i=0}^{r-1} \vG(w+ \frac{i}{r})}{\prod_{i=1}^r \vG(w + v_i')}, 
\qquad v_i' := 1 - \frac{2}{r} - v_i^*, 
\tag{$\ref{eqn:gpf2}$} 
\end{equation}
with summation property $v_1' + \cdot + v_r' = (r-1)/2$. 
Moreover $g(w; \lambda)$ admits a product representation   
\begin{equation}  \label{eqn:GPF-g}
g(w ; \lambda) = D \cdot \delta^{w} \cdot 
\frac{\prod_{i=1}^{r} \sin \pi (w + v_{i}^{*})}{\prod_{i=0}^{r-1} 
\sin \pi (w +\frac{i-1}{r})} \cdot 
\frac{\prod_{i=1}^{r} \vG(w + v_{i}^{*})}{\prod_{i=0}^{r-1} 
\vG(w +\frac{i-1}{r})}, 
\end{equation}  
where $D := C' \cdot x (1-x)^{-a-1/2}$ and 
$\delta := x^{-r}(1-x)^{r-p} = p^p\, (r-p)^{r-p} \, r^{-r}$. 
\end{proposition} 
{\it Proof}. 
Substituting GPF \eqref{eqn:gpf1} and formula \eqref{eqn:Psi-g} into 
equation \eqref{eqn:ratio3} and using $\vG(w+1) = w \vG(w)$, we have 
\begin{align*} 
\frac{f(w +1; \lambda')}{f(w; \lambda')} 
&= \frac{r^r (1-x)^{r-p}}{p^{p} ( r - p )^{r -p} \, x^r} \cdot 
\frac{ \prod_{i=0}^{r-1}(w' + \frac{i - 1}{r})
\prod_{i=0}^{r-1}(w' + \frac{i}{r})}{\prod_{i=0}^{p-1}(w' + \frac{i+a}{p}) 
\prod_{j=0}^{r-p-1}(w' + \frac{j-a}{r-p}) 
\prod_{j=0}^{r-1}(w' + \frac{j-1/2}{r})} 
\cdot 
\frac{\prod_{i=1}^{r}(w' + v_{i})}{\prod_{i=0}^{r-1}(w' + \frac{i}{r})}
\\[1mm] 
&= 1 \cdot 
\frac {\prod_{i = 0}^{r-1}(w' + \frac{i - 1}{r})
\prod_{i=0}^{r}(w' + \frac{i}{r})}{\prod_{i=1}^{r}(w' + v_{i}) 
\prod_{i=0}^{r-p-1}(w' + v_{i}^{*})} \cdot  
\frac{\prod_{i=1}^{r}(w' + v_{i})}{\prod_{i=0}^{r-1}(w' + \frac{i}{r})} 
\\[1mm]
&= \frac{\prod_{i = 0}^{r-1}(w' + \frac{i - 1}{r})}{\prod_{i=1}^{r}(w' + v_{i}^{*})} 
= \frac{\prod_{i=0}^{r-1} \big( -w - \frac{r-1-i}{r}\big)}{ \prod_{i=1}^r (-w-v_i')}  
= \frac{\prod_{i=0}^{r-1}(w+\frac{i}{r})}{\prod_{i=1}^r (w+v_i')},  
\end{align*}
where $Y(x; \bp) = 0$ (Corollary \ref{cor:alg-eq}) and equation \eqref{eqn:division-ast} 
are used in the second equality, while $w' = 2/r-1-w$ and 
$v_i^* = 1-2/r-v_i'$ are used in the fourth equality. 
The ratio representation lifts to GPF \eqref{eqn:gpf2} by Theorem \ref{thm:lift}. 
Summation $v_1' + \cdots + v_r' = (r-1)/2$ is the dual version of 
$v_1 + \cdots + v_r = (r-1)/2$ proved in Proposition \ref{prop:p-p}. 
\par
Substituting GPF \eqref{eqn:gpf2} into formula \eqref{eqn:g(w;l)} with $q = 0$ and $b = 1/2$, we have    
\begin{align*}
g (w ; \lambda) 
&= x (1-x)^{-a-1/2} \cdot \left \{ x^{-r} (1-x)^{r-p} \right\}^w 
\cdot f \left (\ts\frac{2}{r}-w; \lambda' \right) 
\\[1mm]
&= x(1-x)^{-a-1/2} \cdot \delta^w 
\cdot C' \cdot \frac{\prod_{i=0}^{r-1} 
\vG \left (\frac{2}{r} - w + \frac{i}{r} \right)}{\prod_{i=1}^{r}
\vG (\frac{2}{r}-w+v'_i)} 
= D \cdot \delta^{w}\cdot 
\frac{\prod_{i=0}^{r-1}\vG\left(1-(w+\frac{i-1}{r})\right)}{\prod_{i=1}^r \vG(1- (w + v^*_i))} 
\\[2mm]
&= D \cdot \delta^{w} \cdot \frac{\prod_{i=1}^r
\sin \pi (w+v^*_i)}{\prod_{i=0}^{r-1} 
\sin \pi (w+\frac{i-1}{r})} \cdot \frac{\prod^{r}_{i=1}
\vG (w+v_i^*)}{\prod_{i=0}^{r-1}\vG(w + \frac{i-1}{r})},  
\end{align*}
where the last equality follows from the formula 
$\vG(w) \, \vG(1-w) = \pi/(\sin \pi w)$. 
Finally, $\delta := x^{-r}(1-x)^{r-p} = p^p\, (r-p)^{r-p} \, r^{-r}$ 
is again due to $Y(x; \bp) = 0$ in Corollary $\ref{cor:alg-eq}$.  
This proves formula \eqref{eqn:GPF-g}. 
\hfill $\Box$
\subsection{Reciprocity and Connection Formula} \label{ss:r-cf} 
With the help of duality and a connection formula for hypergeometric functions, 
we shall show that reciprocity $\lambda \mapsto \lambda'$ transforms every  
non-elementary solution in $\cI$ into a solution in $\cJ$ (to Problem I);    
see Proposition \ref{prop:GPF-r}.  
\par
Connection formula representing $\hgH(\ba; z)$ in terms of  
$\hgF(\ba; z)$ and $\hgG(\ba; z)$ leads to an expression 
\begin{equation} \label{eqn:cf}
h(w ; \lambda) =C_{f}(w) \, f (w; \lambda) + C_{g}(w) \, g(w; \lambda),  
\end{equation}
as in \cite[formula (64)]{Iwasaki2}, where in the case of 
$\lambda = (p, 0, r; a, 1/2; x) \in \cI$ the coefficients are given by   
$$
C_{f}(w) = \frac{\vG((r-p)w+1/2-a)
\vG(1-r{w})}{\vG(1-a-p{w}) \, \vG(1/2)}, 
\qquad
C_{g}(w) = \frac{\vG((r-p)w+1/2-a)
\vG(r{w}-1)}{\vG((r-p)w-a) \, \vG(r{w}-1/2)}. 
$$
\begin{lemma} \label{lem:h}
If $\lambda = (p, 0, r; a, 1/2; x) \in \cI$ is a non-elementary solution 
with GPF \eqref{eqn:GPF}, then 
\begin{subequations} \label{eqn:h-chi}
\begin{align} 
h(w ; \lambda) &= \vG((r-p) w+1/2-a) \cdot 
(p^p/r^{r})^w \cdot \chi(w) \cdot 
\frac{\prod_{i=0}^{p-1}\vG(w + \frac{i+a}{p})}{\prod_{i=1}^r \vG(w + v_i)}, 
\label{eqn:h} \\[2mm]  
\chi (w) &:= C_3 \cdot \frac{\sin\pi(p w+a)}{\sin(\pi r w) } + C_4 \cdot 
\frac{\prod_{i=1}^r \sin\pi(w+v_i^*)}{\prod_{i=0}^{r-1}\sin\pi(w + \frac{i-1}{r})},   
\label{eqn:chi} 
\end{align} 
\end{subequations} 
where $C_3 :=2^{1/2} \cdot (2\pi)^{(r-p-1)/2} \cdot C \cdot p^{a-1/2} \cdot r^{1/2}$ 
and $C_4 :=(2\pi)^{(r-p-1)/2} \cdot D \cdot (r-p)^{a+1/2} \cdot r^{-1/2}$  
with $C$ and $D$ being the constants in formulas \eqref{eqn:gpf1} and \eqref{eqn:GPF-g}.   
\end{lemma} 
{\it Proof}. 
Reflection formula $\vG(w) \vG(1-w) = \pi/(\sin \pi w)$ and 
Gauss's multiplication formula \eqref{eqn:GMF} allow us to rewrite 
\begin{align*}
C_f(w) 
& =\vG((r-p)w+1/2-a) \cdot \frac{\sin\pi(p{w}+a)}{\sin(\pi r w)}
\cdot \frac{\vG(p{w}+a)}{\Gamma(r w) \, \vG(1/2)} 
\\ 
& =\vG((r-p)w+1/2-a) \cdot (p^p / r^{r})^w \cdot 
C_3' \cdot \frac{\sin\pi(p{w}+a)}{\sin(\pi r w)} \cdot 
\frac{\prod_{i=0}^{p-1}\vG(w+\frac{i+a}{p})}{\prod_{i=0}^{r-1}
\vG(w+\frac{i}{r}) }, 
\end{align*}
where $C_3' := 2^{1/2} \cdot (2\pi)^{(r-p-1)/2} \cdot p^{a-1/2} \cdot r^{1/2}$.  
It is multiplied by GPF \eqref{eqn:gpf1} to yield  
\begin{align} \label{eqn:Cf}
C_f(w) \, f (w;\lambda) = 
\vG ((r-p)w+1/2-a) \cdot (p^p / r^{r})^w \cdot 
C_3 \cdot \frac{\sin \pi (p w+a)}{\sin(\pi r w)} \cdot 
\frac{\prod_{i=0}^{p-1} 
\vG(w + \frac{i+a}{p})}{\prod_{i=0}^{r-1} \vG(w + v_i)}. 
\end{align}
\par
On the other hand, Gauss's multiplication formula \eqref{eqn:GMF} 
and equation \eqref{eqn:v*} lead to  
\begin{align*}
C_g (w) 
&= \vG((r-p)w+1/2-a) \cdot C_4' \cdot (r-p)^{-(r-p)w} \cdot 
\frac{ \prod_{i=0}^{r-1}\vG(w + \frac{i-1}{r})}{\prod_{j=0}^{r-p-1} 
\vG(w + \frac{j-a}{r-p}) \prod_{j=0}^{r-1} \vG(w + \frac{j-1/2}{r})} 
\\[1mm]
&= \vG((r-p)w+1/2-a) \cdot C_4' \cdot (r-p)^{-(r-p)w} \cdot 
\frac{ \prod_{i=0}^{p-1}\vG(w+\frac{i+a}{p}) \prod_{i=0}^{r-1} 
\vG(w+\frac{i-1}{r}) }{\prod_{j=1}^r \vG(w+v_i) \prod_{i=1}^r \vG(w+v_i^*)},  
\end{align*}
where $C_4' := (2\pi)^{(r-p-1)/2}\cdot (r-p)^{a+1/2} \cdot r^{-1/2}$. 
It is multiplied by formula \eqref{eqn:GPF-g} to yield  
\begin{equation} \label{eqn:Cg}
C_{g}(w) \, g(w; \lambda) 
= \vG((r-p)w+1/2-a) \cdot (p^p / r^{r})^w \cdot C_4 \cdot 
\frac{ \prod_{i=1}^{r} \sin \pi (w + v_i^*)}{ \prod_{i=0}^{r-1} 
\sin \pi (w+\frac{i-1}{r})} \cdot \frac{\prod_{i=0}^{r-1} 
\vG(w+\frac{i+a}{p})}{ \prod_{i=1}^{r} \vG(w+v_i)},  
\end{equation}
where $\delta \cdot (r-p)^{-(r-p)} = p^p / r^{r}$ is also used. 
Substituting \eqref{eqn:Cf} and \eqref{eqn:Cg} into \eqref{eqn:cf} 
we obtain \eqref{eqn:h-chi}. \hfill $\Box$
\begin{lemma} \label{lem:chi}
The function $\chi(w)$ in \eqref{eqn:chi} must be a constant, say 
$C_{5}$, so that equation \eqref{eqn:h} becomes 
\begin{equation} \label{eqn:h/G}
\frac{h(w; \lambda)}{\vG((r-p)w+1/2-a)} = C_5 \cdot (p^p / r^{r})^w \cdot 
\frac{\prod_{i=0}^{p-1}\vG(w+\frac{i+a}{p})}{\prod_{i=1}^r \vG(w+v_i)}.
\end{equation}
\end{lemma}
{\it Proof}. 
Let $\tilde{h} (w; \lambda)$ denote the left-hand side of equation \eqref{eqn:h/G}. 
Notice that $\tilde{h}(w; \lambda)$ is an entire function, since the poles of 
$h(w; \lambda)$ are eliminated by those of $\vG((r-p)w+1/2-a)$. 
Formula \eqref{eqn:h} yields 
$$
\chi (w) = (p^p / r^{r})^{-w} \cdot \tilde{h} (w;\lambda) \cdot 
\frac{\prod_{i=1}^r \vG(w + v_i)}{\prod_{i=0}^{p-1}\vG(w + \frac{i+a}{p})},  
$$
which shows that $\chi(w)$ is holomorphic in $\rRe \, w > -\min\{v_i \, \mid \, i = 1, \dots, r \} $, 
while \eqref{eqn:chi} implies that $\chi (w)$ is a periodic function of period $1$.  
Thus $\chi (w)$ is an entire periodic function of period $1$.  
In particular $\chi (w)$ is bounded on the strip $|\rIm \, w| \le 1$. 
Outside the strip, applying the estimate $e^{|\rIm \, w|}/4 \le |\sin \, w|\le e^{|\rIm \, w|}$ 
to formula \eqref{eqn:chi}, we find that $\chi (w)$ is also bounded in $|\rIm \, w| \ge 1$. 
Hence $\chi (w)$ is a bounded entire function, so Liouville's theorem 
forces it to be a constant $C_5$, leading to equation \eqref{eqn:h/G}.  
\hfill $\Box$ \par\medskip
The following lemma will play a key role in the present and next subsections.  
\begin{lemma} \label{lem:dich}
Let $\lambda = (p, 0, r; a, 1/2; x ) \in \cI$ be any non-elementary solution. 
\begin{enumerate}
\setlength{\itemsep}{-1pt}
\item There exists a nonnegative integer $j$ such that either 
\begin{equation} \label{eqn:dich} 
(\rI) \quad a = \frac{p j}{r} \qquad \mbox{or} \qquad 
(\rII) \quad a = \frac{p(j+1/2)}{r}.  
\end{equation}
\item The numbers $a$, $v_1, \dots, v_r$ lie in $\bQ \cap [0, \, 1)$ and 
there exists a division relation   
\begin{equation} \label{eqn:div-a} 
\prod_{i\in \bar{I}_{p}}\left(w + \ts \frac{i+a}{p}\right) \,  \Big| 
\prod_{j\in J_{p}}\left(w + \ts \frac{j - a}{r-p}\right) 
\prod_{j\in J_{q}}\left(w +\ts \frac{j - 1/2}{r}\right), \qquad \bar{I}_p := [p] \setminus I_p.     
\end{equation} 
\end{enumerate}
\end{lemma}
{\it Proof}. 
Assertion (1). 
For multisets of real numbers $S = \{s_1, \dots, s_m\}$ and  
$T = \{ t_1, \dots , t_n \}$ we write $S \succ T$ if $m \le n$ and 
$s_i-t_i \in \bZ_{\ge 0}$ for $i = 1, \dots, m$ after a rearrangement 
of $t_1, \dots, t_n$.   
Recall from \cite[Lemma 5.5]{Iwasaki2} that  
$$
\mbox{if $\dfrac{\vG(w + s_{1}) \cdots \vG(w + s_m)}{ \vG(w + t_1)\cdots \vG(w + t_n)}$ 
is an entire function, then $S \succ T$}. 
$$
Using formula \eqref{eqn:division} in Proposition \ref{prop:p-p} and 
letting $\bar{I}_p := [p]\setminus I_p$, we can rewrite equation \eqref{eqn:h/G} as 
\begin{equation*} \label{eqn:entire}
\frac{h(w;\lambda)}{\vG((r - p)w+1/2-a)} 
= C_{5} \cdot (p^p / r^{r})^{w} \cdot 
\frac{\prod_{i \in \bar{I}_p} \vG (w+\frac{i+a}{p})}{\prod_{j\in J_p}
\vG(w+\frac{j-a}{r - p}) \prod_{j\in J_q} \vG(w+\frac{j-1/2}{r})}.    
\end{equation*}
Since the left-hand side of this equation is an entire function, 
the general fact mentioned above implies 
\begin{equation} \label{eqn:prec}
\left\{ \frac{i + a}{p} \right\}_{i \in \bar{I}_{p}} \succ 
\left\{ \frac{j-a}{r-p} \right\}_{j \in J_{p}} \bigcup \, 
\left \{ \frac{j-1/2}{r} \right \}_{j \in J_q},  
\end{equation}
where the union on the right-hand side is the one as a multiset. 
Recall from Proposition \ref{prop:p-p} that $0 \not\in I_p$ and hence $0 \in \bar{I}_{p}$. 
It then follows from \eqref{eqn:prec} with $i = 0$ that one of the following equations 
holds true:  
\begin{center}
(i) \, $\dfrac{a}{p} = \dfrac{j-a}{r-p}$ \, for some $j\in J_{p} + (r - p) \, \bZ_{\ge 0}$, 
\qquad 
(ii) \, $\dfrac{a}{p} = \dfrac{j-1/2}{r}$ \, for some $j \in J_{q} + r \, \bZ_{\ge 0}$.  
\end{center}
Case (i) directly leads to case (I) of condition \eqref{eqn:dich}. 
In case (ii), Proposition \ref{prop:avoid} shows that   
$0 \not \in J_q$ and hence $j \ge 1$. 
Replacing $j$ by $j + 1$ with $j \ge 0$, we have case (I\!I) of condition \eqref{eqn:dich}. 
\par
Assertion (2). 
It is evident from \eqref{eqn:dich} that $a \in \bQ$ and $a \ge 0$. 
For the same reason we have $a' \ge 0$, since the dual 
$\lambda' \in \cI$ to $\lambda$ is also a non-elementary solution 
by Proposition \ref{prop:GPF-d}. 
By the definition of duality \eqref{eqn:duality} we have $a = 1 - 2p/r-a' \le 1- 2 p/r$. 
Thus $a \in \bQ \cap [0, \, 1)$ and more strongly $0 \le r a /(r-p) \le (r - 2 p)/(r-p) < 1$. 
Hence $r a / (r-p)$ can be an integer only when $a = 0$. 
Accordingly, if $a > 0$ and $j \in J_p$, then Proposition \ref{prop:avoid}  
implies $1 \le j \le r-p-1$ and so $0 < \frac{j-a}{r-p} < 1$ by $a \in (0, \, 1)$.  
If $a = 0$ and $j \in J_p$, then $0 \le j \le r-p-1$ and so   
$0 \le \frac{j-a}{r-p} = \frac{j}{r-p} < 1$. 
In either case we have $0\le \frac{j - a}{r - p}<1$ for any $j \in J_{p}$. 
For any $j \in J_q$, Proposition \ref{prop:avoid} also shows that $1 \le j \le r-1$ 
and so $0< \frac{j - 1/2}{r} <1$. 
Recall from Proposition \ref{prop:p-p} that $I_p$ is a subset of $[p] \setminus \{ 0 \}$. 
Thus, for any $i \in I_p$ we have $1 \le i \le p-1$ and so $0 < \frac{i+a}{p} < 1$. 
Therefore $v_1, \dots, v_r \in \bQ \cap [0, \, 1)$ follows from formula \eqref{eqn:division}. 
Since all elements of the sets in \eqref{eqn:prec} are in $[0, \, 1)$, 
the relation \eqref{eqn:prec} is just an inclusion of multisets   
$$ 
\left\{\frac{i + a}{p}\right\}_{i\in \bar{I}_{p}} \subset 
\left\{\frac{j - a}{r - p}\right\}_{j\in J_{p}} \bigcup 
\left \{\frac{j - 1/2}{r}\right\}_{j\in J_{q}},  
$$
which is equivalent to division relation \eqref{eqn:div-a}. 
This establishes all assertions of the lemma. 
\hfill $\Box$ 
\begin{proposition} \label{prop:GPF-r}
For any non-elementary solution $\lambda =(p, 0 , r; a, 1/2 ;x) \in \cI$ 
there exists a division relation 
\begin{equation} \label{eqn:division2a}
\prod_{i=0}^{p-1} \big(w+ \ts\frac{i+a}{p} \big) \, \Big| \, \ds \prod_{i=1}^r (w+v_i),  
\tag{$\ref{eqn:division2}$} 
\end{equation} 
which allows us to arrange $v_1, \dots, v_r$ so that  
$\prod_{i=0}^{p-1}\big(w + \ts \frac{i+a}{p} \big) = \prod_{i=r-p+1}^r (w + v_i)$.    
Then we have 
\begin{alignat}{2}
h(w; \lambda) 
&= C_6 \cdot \delta^w \cdot   
\dfrac{\prod_{j=1}^{r-p-1} \vG\big( w + c + \frac{j}{r-p} \big)}{ \prod_{i=1}^{r-p} \vG(w + v_i)}, 
\qquad & c &:= \frac{1/2-a}{r-p},  
\label{eqn:gpf-h}  \\[1mm]
f(w; \check{\lambda}) 
&= \check{C} \cdot \dfrac{\prod_{i=0}^{r-p-1} 
\vG\big(w + \frac{i}{r-p}\big)}{ \prod_{i=1}^{r-p} \vG( w + \check{v}_i) }, 
\qquad & \check{v}_i &:= v_i -c.  \tag{$\ref{eqn:gpf-r}$} 
\label{eqn:gpf-r2}
\end{alignat}
where $\delta := p^p (r-p)^{r-p} r^{-r}$ while $C_6$ and $\check{C}$ 
are nonzero constants. 
\end{proposition}  
{\it Proof}.  
The division relation \eqref{eqn:division2} is obtained by 
multiplying both sides of relation \eqref{eqn:div-a} by  
$\prod_{i \in I_p} (w + \frac{i+a}{p})$ and using equation \eqref{eqn:division}. 
With the convention $\prod_{i=0}^{p-1}\big(w + \ts \frac{i+a}{p} \big) = \prod_{i=r-p+1}^r (w + v_i)$, 
formula \eqref{eqn:h/G} gives   
\begin{equation*} 
h(w;\lambda)=C_{5} \cdot (p^p / r^{r})^{w} \cdot 
\frac{\vG((r-p)w+1/2-a)}{\prod_{i=1}^{r - p} \vG(w + v_i)},  
\end{equation*}
which is converted to formula \eqref{eqn:gpf-h} by applying Gauss's multiplication 
formula \eqref{eqn:GMF} to $\vG((r-p) w + 1/2-a)$, where 
$C_6 := C_5 \cdot (2 \pi)^{(p+1-r)/2} (r-p)^{-a}$. 
From formula \eqref{eqn:h(w;l)} with $q = 0$ and $b = 1/2$, we have
$$
f (w; \check{\lambda}) = x^{r (w-c) - 1}
(1 - x)^{(p-r) (w-c)+a+1/2} \, h(w-c ; \lambda). 
$$
Formula \eqref{eqn:gpf-r2} is then obtaind by substituting formula \eqref{eqn:gpf-h} 
into this equation and using $x^r (1-x)^{p-r} \cdot \delta = 1$, which follows from 
$Y(x, \bp) = 0$ in Corollary \ref{cor:alg-eq}. \hfill $\Box$ \par\medskip
{\it Proof of Theorem $\ref{thm:gpf}$}. 
As for assertion (1) of Theorem \ref{thm:gpf}, 
GPF \eqref{eqn:gpf1} is proved in Proposition \ref{prop:p-p} 
and GPF \eqref{eqn:gpf2} is proved in Proposition \ref{prop:GPF-d}. 
GPF \eqref{eqn:gpf-r2} in assertion (2) is established in Proposition \ref{prop:GPF-r}. 
\hfill $\Box$ 
\subsection{Duality in More Depth} \label{ss:dual-r}
In Lemma \ref{lem:dich}, in parallel with conditions $(\rI)$ and $(\rII)$ 
for a piece of data $\lambda$, 
we can also speak of conditions $(\rI')$ and $(\rII')$ for its dual $\lambda'$.     
Thus, for a dual pair $(\lambda, \lambda')$, there are a total of four possibilities  
$(\rI, \rI')$, $(\rII, \rII')$, $(\rII, \rI')$, $(\rI, \rII')$, but they are essentially three, 
because $(\rI, \rII')$ is reduced to $(\rII, \rI')$ by exchanging the roles of  
$\lambda$ and $\lambda'$. 
Lemma \ref{lem:dich} becomes more useful, when applied 
to the dual pair $(\lambda, \lambda')$.    
\begin{lemma} \label{lem:cand} 
For any non-elementary solution $\lambda = (p, 0, r; a, 1/2; x) \in \cI$ and its dual   
$\lambda' = (p, 0, r; a', 1/2; x) \in \cI$ there exist nonnegative integers $j$ and $j'$ 
such that the triple $(r/p, a, a')$ takes one of the three forms,   
\begin{subequations} \label{eqn:candidate}
\begin{alignat}{4}
& (\rI, \, \rI'),  \qquad & \frac{r}{p} &= j + j' + 2, \qquad &  a &= \frac{j}{j + j' + 2}, \qquad &  a' &= \dfrac{j'}{ j + j' + 2},  
\label{eqn:candidate1} \\[1mm]
& (\rII, \rII'),  \qquad & \frac{r}{p} &= j + j' + 3, \qquad &  a &= \dfrac{2 j + 1}{2(j + j' + 3)}, \qquad & a' &= \dfrac{2 j' + 1}{2(j + j' + 3)}, 
\label{eqn:candidate2} \\[1mm]
& (\rII, \, \rI'),  \qquad & \frac{r}{p} &= j + j' + \frac{5}{2}, \qquad & a &= \dfrac{2 j + 1}{2(j + j') + 5}, \qquad &  a' &= \dfrac{2 j'}{2(j + j') + 5}, 
\label{eqn:candidate3}  
\end{alignat}
\end{subequations}  
where $\lambda$ and $\lambda'$ may be swapped in case \eqref{eqn:candidate3}.        
\end{lemma}
{\it Proof}.   
The main part of the definition of duality \eqref{eqn:duality} is the equation  
\begin{equation} \label{eqn:a-ap}
a + a' = 1 - 2 p/r.
\end{equation}
In case $(\rI, \rI')$, Lemma \ref{lem:dich} implies that there exist 
nonnegative integers $j$ and $j'$ such that $a = p j/r$ and $a' = p j'/r$.  
Substituting these into equation \eqref{eqn:a-ap} yields $r/p = j + j' + 2$,  
$a = j/(j+j'+2)$ and $a' = j'/(j+j'+2)$.  
This is case \eqref{eqn:candidate1}. 
In case $(\rII, \rII')$, Lemma \ref{lem:dich} implies that there exist 
nonnegative integers  $j$ and $j'$ such that $a = p (j + 1/2)/r$ and 
$a' = p(j' + 1/2)/r$. 
Substituting these into equation \eqref{eqn:a-ap} yields $r/p = j + j' +3$, 
$a = (2 j+1)/(2 j+2 j'+6)$ and $a' = (2 j'+1)/(2 j+2 j'+6)$.  
This is case \eqref{eqn:candidate2}. 
In case $(\rII, \rI')$, Lemma \ref{lem:dich} implies that there exist 
nonnegative integers $j$ and $j'$ such that and $a = p(j + 1/2)/r$ and $a' = p j'/r$ .  
Substitutitng these into equation \eqref{eqn:a-ap} yields $r/p = j + j' +5/2$, 
$a = (2 j + 1)/(2 j + 2 j' + 5)$ and $a' = 2 j'/(2 j + 2  j'+ 5)$. 
This is case \eqref{eqn:candidate3}. \hfill $\Box$ \par\medskip
The next step is to show that neither \eqref{eqn:candidate2} nor  
\eqref{eqn:candidate3} can occur.    
To this end, we make a calculation similar to that in \S \ref{ss:thgs}. 
Euler's transformation in \eqref{eqn:euler} allows us to rewrite 
the definitions \eqref{eqn:Phi} and \eqref{eqn:P} as   
\begin{subequations} \label{eqn:PhiP-e}
\begin{align} 
\Phi (w ; \lambda)
&= (r w)_{r-1} \cdot \langle (1-z)^{r-p} \cdot \hgF( \bal(w); z) \cdot 
\hgF(\bv - \bal(w+1) ; z) \rangle_{r -1} \big|_{z = x}, 
\label{eqn:Phi-e} \\ 
P(w; \lambda) 
&= (r w)_{r} \cdot 
\langle (1-z)^{r-p-1} \cdot \hgF(\bal(w) ; z) \cdot 
\hgF(\be_3 - \bal(w+1); z) \rangle_{r - 1} \big|_{z=x},  
\label{eqn:P-e} 
\end{align}
\end{subequations}    
respectively, where $\bal(w) := ( p w + a, 1/2; r w)$, $\bv := (1, 1, 2)$ and $\be_3 := (0, 0, 1)$. 
\begin{lemma} \label{lem:thgs-p}
Given a piece of integral data $\lambda = (p, 0, r; a; 1/2; x) \in \cI$, 
we put $\zeta_i := -(a + i)/p$ for $i \in [p]$. 
If $r \zeta_i$ is not an integer, then we have  
\begin{subequations} \label{eqn:thgs-p}
\begin{align} 
\Phi(\zeta_i; \lambda) 
&= (-1)^{p-1-i} (1-x)^{r-p} \cdot (r \zeta_i + i)_{r-p} \cdot 
\cF_i(1/2; r \zeta_i; x) \cdot \cF_{p-1-i}(1/2; 2-r(\zeta_i+1); x), 
\label{eqn:thgs-p1} \\
P(\zeta_i; \lambda) 
&= (-1)^{p-i} (1-x)^{r-p-1} \cdot (r \zeta_i + i)_{r-p} \cdot \cF_i(1/2; r \zeta_i; x) 
\cdot \cF_{p-i}(-1/2; 1-r(\zeta_i+1); x).   
\label{eqn:thgs-p2} 
\end{align}
\end{subequations}
\end{lemma}
{\it Proof}. 
Substituting $w = \zeta_i$ in equations \eqref{eqn:PhiP-e} and observing that 
\begin{align*}
\bal(\zeta_i) &= (-i, \, 1/2, \, r \zeta_i), \\ 
\bv-\bal(\zeta_i +1) &= (-(p-1-i), \, 1/2, \, 2-r(\zeta_i+1)), \\
\be_3-\bal(\zeta_i +1) &= (-(p-i), \, -1/2, \, 1 -r(\zeta_i+1)), \\
(r \zeta_i)_{r-1} &= (-1)^{p-1-i} \cdot  (r \zeta_i + i)_{r-p} \cdot (r \zeta_i)_i \cdot (2-r( \zeta_i+1))_{p-1-i}, \\
(r \zeta_i)_r &= (-1)^{p-i}  \cdot (r \zeta_i + i)_{r-p} \cdot (r \zeta_i)_i \cdot (1-r( \zeta_i+1))_{p-i},  
\end{align*}
with $r \zeta_i \not \in \bZ$, $2-r(\zeta_i+1) \not \in \bZ$ and 
$1-r(\zeta_i+1) \not \in \bZ$, 
we obtain formulas \eqref{eqn:thgs-p} from definition \eqref{eqn:cF}. 
\hfill $\Box$
\begin{lemma} \label{lem:cand2}
In Lemma $\ref{lem:cand}$ neither \eqref{eqn:candidate2} nor 
\eqref{eqn:candidate3} can occur.  
In \eqref{eqn:candidate1} we have $p(j + j' + 1) \equiv 0 \bmod 2$. 
\end{lemma}
{\it Proof}. 
Since $\lambda$ is a non-elementary solution in $\cI$, 
we have $\Phi(w; \lambda) \equiv 0$ by Proposition \ref{prop:Phi-P}. 
Division relation \eqref{eqn:division2a} in Proposition \ref{prop:GPF-r} tells us 
that $P(\zeta_i; \lambda) = 0$ for every $i \in [p]$. 
In cases \eqref{eqn:candidate2} and \eqref{eqn:candidate3} we have 
$r \zeta_0 = -(r/p) \, a = -j - 1/2$, which is not an integer. 
So formulas \eqref{eqn:thgs-p} with $i = 0$ are available, yielding   
\begin{align*} 
\Phi(\zeta_0; \lambda) 
&= (-1)^{p-1} (1-x)^{r-p} \cdot (r \zeta_0)_{r-p} \cdot \cF_{p-1}(1/2; 2-r(\zeta_0 +1); x) = 0, \\
P(\zeta_0; \lambda) 
&= (-1)^{p} (1-x)^{r-p-1} \cdot (r \zeta_0)_{r-p}  \cdot \cF_{p}(-1/2; 1-r(\zeta_0 +1); x) = 0.    
\end{align*}
Since $(r \zeta_0)_{r-p}$ is nonzero, we have $\cF_{p-1}(1/2; 2-r(\zeta_0 +1); x) 
= \cF_{p}(-1/2; 1-r(\zeta_0 +1); x) = 0$.  
Assertion (3) of Lemma \ref{lem:thgs} then forces $\cF_{p}(-1/2; 1 -r(\zeta_0 + 1); z) \equiv 0$. 
However, by assertion (1) of the same lemma we have $\cF_{p}(-1/2; 1 -r(\zeta_0 + 1); z) \not \equiv 0$, 
as $-1/2$ and $1-r(\zeta_0 + 1)$ are not integers. 
This contradiction shows that neither \eqref{eqn:candidate2} nor \eqref{eqn:candidate3} is feasible.   
In case \eqref{eqn:candidate1} we have $r - p = p (j + j' + 2) - p = p(j + j' +1) \equiv 0 \bmod 2$ 
by the congruence $r \equiv p \bmod 2$ in formula \eqref{eqn:integral}. 
\hfill $\Box$ \par\medskip
{\it Proof of Theorem $\ref{thm:a-x}$}. 
Assertion (1) is already proved in Proposition \ref{prop:GPF-d}. 
Assertion (2) follows from Lemmas \ref{lem:cand} and \ref{lem:cand2}. 
Now that $s := r/p \in \bZ_{\ge 2}$, the equation $Y(x; \bp = 0)$ in Corollary \ref{cor:alg-eq} 
is equivalent to $\{(s-1)^{s-1} x^s \}^p = \{ s^s (1-x)^{s-1} \}^p$. 
We can take the $p$-th root of it to obtain algebraic equation 
$(s-1)^{s-1} x^s = s^s (1-x)^{s-1}$ in \eqref{eqn:alg-eq}, since  
$(s-1)^{s-1} x^s$ and $s^s (1-x)^{s-1}$ are positive.  
This proves assertion (3). \hfill $\Box$ 
\subsection{Necessary and Sufficient Condition} \label{ss:nsc} 
With Theorem \ref{thm:a-x} in hand, we can now add one more ingredient to 
obtain a necessary and sufficient condition for a given piece of data 
$\lambda \in \cI$ to be a non-elementary solution. 
To state the result, we set  
$$
\chi(j, k; s) := - \frac{1}{s} \left\{ j + ( s-1) \left( k-\frac{1}{2} \right) \right\}, 
\quad \mbox{for} \quad  j, \, k \in \bZ_{\ge 0}, \,\, s \in \bZ_{\ge 2}.  
$$ 
\begin{theorem} \label{thm:nsc} 
Let $\lambda = (p, 0, r; a, 1/2; x) \in \cI$ be a piece of integral data and 
$\lambda' = (p, 0, r; a', 1/2; x) \in \cI$ its dual. 
Then $\lambda$ and $\lambda'$ form a dual pair of non-elementary solutions 
if and only if  
\begin{enumerate}
\setlength{\itemsep}{-1pt}
\item the ratio $s := r/p$ is an integer such that $s \ge 2$ and $p(s-1) \equiv 0 \bmod 2$,  
\item the argument $x$ satisfies the algebraic equation
$(s-1)^{s-1} x^s - s^s (1-x)^{s-1} = 0$, 
\item there exist nonnegative integers $j$ and $j'$ with $ j + j' = s-2$ such that
$a = j/s$,  $a' = j'/s$ and  
\begin{equation} \label{eqn:nsc} 
\cF_k \left( \chi(j, k; s); 1/2 - k; x \right)  \cdot 
\cF_{k'} \left( \chi(j', k'; s); 1/2 - k'; x \right) = 0,   
\end{equation}
for all nonnegative integers $k$ and $k'$ with $k + k' = r-1$, 
where $\cF_k(\beta; \gamma; z)$ is defined in formula \eqref{eqn:cF}.   
\end{enumerate} 
\end{theorem} 
{\it Proof}. 
All conditions other than equation \eqref{eqn:nsc} come from Theorem \ref{thm:a-x}, 
so the main part of the proof lies in proving condition \eqref{eqn:nsc}.  
We recall from Lemma \ref{lem:sol-q} that $\eta_k := -(k-1/2)/r$, 
$c_k := (r-p) \eta_k - a$ and $d_k := 1 - (r-p) (\eta_k + 1) + a$. 
Put $c_k' := (r-p) \eta_k - a'$ and $k' := r-1-k$. 
Observe that $\eta_k + 1 = 2/r - \eta_{k'}$ and
$$
d_k = 1 -(r-p) (2/r - \eta_{k'}) + a = (r-p) \eta_{k'} -(1 - 2p/r - a) 
= (r-p) \eta_{k'} - a' = c'_{k'}. 
$$   
From formula \eqref{eqn:saap} in Theorem \ref{thm:a-x} we have $a = j/s$ and 
$a' = j'/s$ with $j + j' = s-2$, and hence  
$$
c_k = - (r-p) (k-1/2)/r - j/s = - (s-1) (k-1/2)/s -j/s = \chi(j, k; s). 
$$
Similarly we have $d_k = c'_{k'} = \chi(j', k' ; s)$.  
Therefore formula \eqref{eqn:sol-q1} in Lemma \ref{lem:sol-q} becomes 
\begin{equation} \label{eqn:sol-q1a}
\Phi(\eta_k; \lambda) = (-1)^{k'} \cdot 
\cF_k \left( \chi(j, k; s); 1/2 - k; x \right) \cdot 
\cF_{k'} \left( \chi(j', k'; s); 1/2 - k'; x \right) \quad \mbox{for any} \quad k \in [r].   
\end{equation}
By Corollary \ref{cor:A(w;l)} and Lemma \ref{lem:Phi-P}, 
$\Phi(w; \lambda)$ is a polynomial of degree at most $r-1$. 
Hence $\Phi(w; \lambda)$ vanishes identically if and only if 
$\Phi(\eta_k; \lambda) = 0$ for every $k \in [r]$, since 
the $r$ numbers $\eta_k$, $k \in [r]$, are mutually distinct.  
On the other hand, by Proposition \ref{prop:Phi-P}, $\lambda$ 
is a non-elementary solution if and only if $\Phi(w; \lambda)$ vanishes identically. 
Accordingly, the necessary and sufficient condition \eqref{eqn:nsc} 
follows from formula \eqref{eqn:sol-q1a}.  
\hfill $\Box$ 
\section{Degree of the Argument} \label{sec:deg}
The aim of this section is to establish Theorems \ref{thm:degree} and \ref{thm:rqa}. 
It is convenient to transform algebraic equation \eqref{eqn:alg-eq} into 
a monic equation with integer coefficients, whose roots are algebraic integers.          
\begin{lemma} \label{lem:algequ} 
For any non-elementary solution $\lambda = (p, 0, r; a, 1/2; x) \in \cI$ the argument 
$x$ can be represented as $x = s/(y + s)$, where $s := r/p \in \bZ_{\ge 2}$ and $y$ 
is the unique positive root, which is simple, of the algebraic equation 
\begin{equation} \label{eqn:algeq}
\psi_s(z) := z^{s-1} (z + s) - (s-1)^{s-1} = 0.  
\end{equation}
\end{lemma}
{\it Proof}. 
It is straightforward to see that change of variable $z = s/(w + s)$ transforms 
equation \eqref{eqn:alg-eq} into a monic equation $\psi_s(w) = 0$.  
Since $\psi_s(0) = -(s-1)^{s-1} < 0$ and $\psi_s'(z) = s \, z^{s-2} (z + s-1) > 0$ 
for $z > 0$, equation \eqref{eqn:algeq} has a unique positive root 
$y$ with derivative $\psi_s'(y) > 0$, so $y$ is a simple root.     
\hfill $\Box$ \par\medskip
When $s$ is odd, equation \eqref{eqn:algeq} has a double root at $z = 1-s$, 
which we think is {\sl trivial}; any other root is said to be {\sl nontrivial}.   
When $s$ is even, any root is {\sl nontrivial}. 
The main result of this section is the following.  
\begin{theorem} \label{thm:deg} 
For each integer $s \ge 3$ any montrivial root $\alpha$ of equation \eqref{eqn:algeq} 
has degree 
\begin{equation} \label{eqn:deg}
\deg \alpha > \delta(s) := \frac{(\log p_s)(s-1)}{1 + \log(s-1)}, 
\end{equation}
where $p_s$ is the least prime factor of the integer $s-1$.   
\end{theorem} 
\begin{remark} \label{rem:deg} 
If $s$ is odd then $p_s = 2$, whereas if $s$ is even then $p_s$ is an odd prime  
greater than or equal to $3$.  
In Theorems \ref{thm:deg} and \ref{thm:degree}, it is practical to 
employ, in place of $\delta(s)$, the lower bounds 
\begin{equation} \label{eqn:delta}
\delta_0(s) := \frac{(\log 3) (s-1)}{1 + \log(s-1)} \quad \mbox{when $s$ is even}; \qquad 
\delta_1(s) := \frac{(\log 2) (s-1)}{1 + \log(s-1)} \quad \mbox{when $s$ is odd}.   
\end{equation}
We remark that $\delta_0(s)$ and $\delta_1(s)$ are strictly increasing 
functions in $s \ge 2$ tending to infinity as $s \to \infty$.       
\end{remark}
\par
The proof of Theorem \ref{thm:deg} as well as that of Theorem \ref{thm:degree} 
is given at the end of \S \ref{ss:arith}. 
\subsection{Real and Complex Roots} \label{ss:rcr} 
We investigate the real and complex roots of the polynomial $\psi_s(z)$. 
We begin with the real roots. 
\begin{lemma} \label{lem:realr} 
The real roots of $\psi_s(z)$ are located in the following manner. 
\begin{enumerate}
\setlength{\itemsep}{-1pt}
\item If $s \ge 2$ is even, then there are one positive root and one negative root, both simple.  
\item If $s \ge 3$ is odd, then there are one positive root, simple, 
and one negative root $z = 1-s$, double.  
\end{enumerate}
\end{lemma} 
{\it Proof}. 
First we find $\psi_s'(z) = s z^{s-2}(z + s -1)$. 
Since $\psi_s'(z) > 0$ for $z > 0$ and $\psi_s(0) = -(s-1)^{s-1} < 0 < \psi_s(s-1) = 2 (s-1)^s$, 
the polynomial $\psi_s(z)$ has a unique positive root, which is simple. 
When $s$ is even,  since $\psi_s'(z) < 0$ for $z < 1-s$, $\psi_s'(z) > 0$ for $z \in (1-s, \, 0)$ 
and $\psi_s(-s-1) = (s+1)^{s-1} - (s-1)^{s-1} > 0 > \psi_s(1-s) = -2 (s-1)^{s-1}$ 
and $\psi_s(0) = -(s-1)^{s-1} < 0$, $\psi_s(z)$ has a unique negative root, 
which is simple and lies in $(-s-1, \, 1-s)$. 
When $s$ is odd, since $\psi_s'(z) > 0$ for $z < 1-s$, $\psi_s'(z) < 0$ for $z \in (1-s, \, 0)$ 
and $\psi_s(1-s) = 0 > \psi_s(0) = -(s-1)^{s-1}$, $\psi_s(z)$ has a unique negative root 
at $z = 1-s$, which is double, because $\psi_s''(1-s) = s (1-s)^{s-2} < 0$.   
\hfill $\Box$ \par\medskip
Change of variable $z = -s - t$ transforms equation \eqref{eqn:algeq} into 
\begin{equation} \label{eqn:gs}
u_s(t) = (-1)^s \qquad \mbox{with} \qquad 
u_s(t) := t \left(\frac{s + t}{s-1}\right)^{s-1}. 
\end{equation}
Basically the parameter $s$ takes integer values, but when doing calculus 
with respect to $s$, we allow for $s$ to be a real parameter greater than $1$.    
We shall often use the inequality  
\begin{equation} \label{eqn:log}  
\log(1 + t) < t \quad \mbox{for any $t > -1$ with equality  only when $t = 0$}.   
\end{equation}
\begin{lemma} \label{lem:c(s)} 
For each $s > 1$ the equation $u_s(t) = 1$ has a unique solution $t = c(s)$ 
in $t \ge -1$, which is simple and lies in the interval $(0, \, 1)$.    
The solution $c(s)$ is a $C^{\infty}$-function of $s > 1$, which is strictly 
decreasing in $s > 1$ and converges to a constant $c_0$ as $s \to \infty$, 
where $c_0 \approx 0.278465$ is the unique real number such that 
\begin{equation*} \label{eqn:c}
c_0 \, \re^{c_0 + 1} = 1, \qquad 0 < c_0 < 1. 
\end{equation*}
In particular, when $s \ge 2$ is an even integer, the negative root of 
equation \eqref{eqn:algeq} lies at $z = -s- c(s)$. 
\end{lemma} 
{\it Proof}. 
The number $c_0$ exists uniquely, because the function 
$v(t) := t \, \re^{t+1}$ satisfies $v(0) = 0 < 1 < v(1) = \re^2$ and 
$v'(t) = (t + 1) \, \re^{t+1} > 0$ for any $t \in (0, \, 1)$. 
A numerical computation shows $c_0 \approx 0.278465$. 
Since 
$$
u_s'(t) = \frac{s (t+1)}{s-1} \left( \frac{s + t}{s-1} \right)^{s-2} > 0 
\qquad \mbox{for} \quad t > -1, 
$$
and $u_s(0) = 0 < 1 < u_s(1) = \{ (s+1)/(s-1) \}^{s-1}$, 
the equation $u_s(t) = 1$ has a unique solution $t = c(s)$ in $t \ge -1$, which is simple 
and actually lies in $(0, \, 1)$.   
Thinking of $u(s, t) := u_s(t)$ as a real $C^{\infty}$-function of two variables 
$(s, t) \in  (1, \, \infty) \times (0, \, 1)$, we find that the partial derivatives of 
$u(s, t)$ are given by 
$$
\dfrac{\partial u}{\partial s} = - u(s, t) \cdot \left\{ \frac{1 + t}{s + t} + 
\log \left(1 - \frac{1 + t}{s + t} \right) \right\}, \qquad 
\dfrac{\partial u}{\partial t} = u(s, t) \cdot \frac{s(1 + t)}{t( s + t)} > 0.   
$$
As $u(s, c(s)) = 1$, the implicit function theorem implies that $c(s)$ is a 
$C^{\infty}$-function of $s > 1$ with derivative   
$$
c'(s) = \frac{c(s) \, \{ s + c(s) \} }{s \, \{ 1 + c(s) \}} 
\left\{ \frac{1 + c(s)}{s + c(s)} + \log \left( 1 - \frac{1 + c(s) }{s + c(s)} \right) \right\} < 0,    
$$
where the negativity of $c'(s)$ follows from estimate \eqref{eqn:log}.    
Thus the function $c(s) \in (0, \, 1)$ is strictly decreasing in $s > 1$ and converges 
to a number $c_1 \in [0, \, 1)$ as $s \to \infty$. 
Since the function $u_s(t) = t (1 + t/s)^{s-1}(1-1/s)^{1-s}$ converges to $v(t) = t \, \re^{t+1}$ 
uniformly on $[0, \, 1]$ as $s \to \infty$, we have $v(c_1) = 1$ and $0 \le c_1 < 1$, actually 
$c_1 \neq 0$.  
By the uniqueness of such a number we have $c_1 = c_0$. 
Hence $c(s)$ converges to $c_0$ as $s \to \infty$. 
\hfill $\Box$ \par\medskip
All (complex) roots of equation \eqref{eqn:algeq} except for the negative 
one admit the following estimate. 
\begin{proposition} \label{prop:s-1}
Any root $\alpha$ of equation \eqref{eqn:algeq} except for the negative 
one satisfies $|\alpha| < s-1$. 
\end{proposition}
{\it Proof}. 
The case where $s$ is even is proved in Lemma \ref{lem:fsp1} below,  
while the odd case is proved in Lemma \ref{lem:fsm1}. 
\hfill $\Box$ \par\medskip
In order to formulate Lemmas \ref{lem:fsp1} and \ref{lem:fsm1}, it is 
convenient to rewrite equation \eqref{eqn:algeq} as   
\begin{equation} \label{eqn:fs}
F_s(z) = (-1)^s \qquad \mbox{with} \qquad    
F_s(z) := - \frac{(-z)^{s-1} (z + s)}{(s-1)^{s-1}}.   
\end{equation}
The polynomial $F_s(z)$ has exactly two critical points; 
one is $z = 1-s$ with critical value $-1$ and the other 
is $z = 0$ with critical value $0$, since the derivative of $F_s(z)$ is calculated as   
\begin{equation} \label{eqn:fsd}
F_s'(z) = \frac{s (-z)^{s-2}(z + s -1)}{(s-1)^{s-1}}.    
\end{equation}
\par
Let $C := \{ z \in \bC \mid |z| = s-1 \}$ be the circle of radius $s-1$ 
with center at the origin.   
We have
\begin{equation} \label{eqn:fse} 
|F_s(z)| \ge 1 \quad \mbox{for any $z \in C$ with equality only when $z = 1-s$}.    
\end{equation} 
Indeed, for any $z \in C$ we have $\rRe(z) \ge 1-s$ and hence       
$$
|F_s(z))| 
= \frac{|z|^{s-1} |z + s| }{(s - 1)^{s-1}} = | z + s | \ge \rRe(z) + s \ge 1 
\quad \mbox{with equalities only when $z = 1-s$}.     
$$
Let $D := \{ z \in \bC \mid |z| < s-1 \}$ and $E := \{ z \in \bC \mid |z| > s-1 \}$. 
Denote their closures by $\overline{D}$ and $\overline{E}$.      
\begin{lemma} \label{lem:fsp1} 
Any root $\alpha$ of the equation $F_s(z) = 1$ other than 
$z = -s -c(s)$ satisfies $|\alpha| < s-1$, where $c(s)$ is the function 
given in Lemma $\ref{lem:c(s)}$.    
\end{lemma}
{\it Proof}. 
It is evident from definition \eqref{eqn:fs} that the equation 
$F_s(z) = 0$ has only one root $z = 0$ of multiplicity $s-1$ in $D$.    
For any $w \in \bD := \{ w \in \bC \mid |w| < 1 \}$, estimate \eqref{eqn:fse} 
shows that $|w| < 1 \le |F_s(z)|$ for all $z \in C$.  
Rouch\'{e}'s theorem implies that the algebraic equation $F_s(z) = w$ (of 
degree $s$) has exactly $s-1$ roots in $\overline{D}$ counted with multiplicity.   
The remaining one root, say $z = g(w)$, is simple and lies in $E$.   
By the inverse function theorem $g : \bD \to E$, $w \mapsto g(w)$ 
is a holomorphic function. 
Let $b$ be any root of $F_s(z) = 1$ in $E$.  
Since $F_s'(b) \neq 0$, the inverse function theorem gives a holomorphic 
function $h_b : U_b \to V_b$ such that $h_b(1) = b$ and $F_s(h_b(w)) = w$ for any $w \in U_b$,  
where $U_b$ is an open set in $\bC$ with $1 \in U_b$, while $V_b$ is an open set 
with $b \in V_b \subset E$.  
By the uniqueness of $g(w)$ we have $g(w) = h_b(w)$ for any $w \in \bD \cap U_b$. 
For the same reason we have $g(w) = h_a(w)$ for any $w \in \bD \cap U_a$ 
with $a := -s-c(s)$.   
Thus $h_b(w) = h_a(w)$ for any $w \in \bD \cap U_a \cap U_b$. 
Letting $\bD \cap U_a \cap U_b \ni w \to 1$, we have $b = h_b(1) = h_a(1) = a$.  
Therefore $z = -s- c(s)$ is the only root of $F_s(z) = 1$ in $\overline{E}$.    
\hfill $\Box$ 
\begin{lemma} \label{lem:fsm1} 
Any root $\alpha$ other than $z = 1-s$ 
of the equation $F_s(z) = -1$ satisfies $|\alpha| < s-1$.   
\end{lemma}
{\it Proof}.  
Suppose that the equation $F_s(z) = -1$ has a root $a \in E$.       
Since $F_s'(a) \neq 0$, the inverse function theorem gives a holomorphic function 
$h : U \to V$ such that $h(-1) = a$ and $F_s(h(w)) = w$ for any $w \in U$,  
where $U$ is an open set in $\bC$ with $-1 \in U$, while $V$ is 
an open set with $a \in V \subset E$. 
Notice that $g(w) = h(w)$ for any $w \in \bD \cap U$, where 
$g : \bD \to E$ is the same function as that in the proof of 
Lemma \ref{lem:fsp1}. 
Put $U^* := \{ \bar{w} \mid w \in U \}$ and $V^* := \{ \bar{z} \mid z \in V \}$. 
The complex conjugate $\bar{a}$ to $a$ is also a root of  $F_s(z) = -1$  
and the holomorphic function $h^* : U^* \to V^*$ defined by 
$h^*(w) := \overline{h(\bar{w})}$ satisfies $h^*(-1) = \bar{a}$ and 
$F_s(h^*(w)) = w$ for any $w \in U^*$, thus   
$g(w) = h^*(w)$ for any $w \in \bD \cap U^*$, and hence  
$h(w) = h^*(w)$ for any $w \in \bD \cap U \cap U^*$. 
Letting $\bD \cap U \cap U^* \ni w \to -1$, we have 
$a = h(-1) = h^*(-1) = \bar{a}$, which means that $a$ is a real 
root of $F_s(z) = -1$ in $z < 1-s$.   
However, there is no such root, because formulas \eqref{eqn:fs} and \eqref{eqn:fsd} 
imply $F_s(1-s) = -1$ and $F_s'(z) < 0$ for any $z <1-s$. 
This contradiction shows that there is no root of $F_s(z) = -1$ in $E$ and 
hence in $\overline{E} \setminus \{ 1-s \}$.    
\hfill $\Box$
\subsection{Radial and Angular Equations} \label{ss:rae}
We write equation \eqref{eqn:algeq} in the form $z^{s-1} (z + s) = (s-1)^{s-1}$. 
In its polar representation, putting aside its angular component for the moment, 
we consider its radial component $|z|^{s-1} |z + s| = (s-1)^{s-1}$, which reads   
\begin{equation} \label{eqn:abs}
r^{s-1} \rho = (s-1)^{s-1} \qquad \mbox{with} \quad 
\rho = \rho(r, \theta) := |z + s| = \sqrt{s^2 + 2 s r \cos \theta + r^2}   
\end{equation}
in terms of $z = r \re^{\ri \theta}$ with $r > 0$ and $\theta \in \bR$. 
The location of the unique negative root of equation \eqref{eqn:algeq} is 
already known to us, so Proposition \ref{prop:s-1} restricts our 
attention to $0 < r \le s-1$ and $-\pi < \theta < \pi$, where $\theta$ 
may further be restricted to $0 \le \theta < \pi$ up to complex conjugation.   
We take the logarithm of equation \eqref{eqn:abs} to have      
\begin{equation} \label{eqn:re}
h(r, \theta) := \log \rho(r, \theta) + (s-1) \log r - (s-1) \log(s-1) = 0. 
\end{equation}
\begin{lemma} \label{lem:radial} 
For each $\theta \in [0, \, \pi)$ there exists a unique solution $r = r(\theta)$ 
to equation \eqref{eqn:re} such that $0 < r < s-1$.  
As a function of $\theta$, the solution $r = r(\theta)$ is strictly increasing in 
$\theta \in [0, \, \pi)$ and $r(\theta) \to s-1$ as $\theta \to \pi-0$.  
\end{lemma}
{\it Proof}. 
Notice that $\rho(r, \theta) = \sqrt{(s-r)^2 + 2 r s (1 + \cos \theta)} > s-r \ge 1$ 
for any $(r, \theta) \in (0, \, s-1] \times [0, \, \pi)$ and 
$$
\lim_{r \to + 0} h(r, \theta) = - \infty, \quad  
h(s-1, \theta) = \log \rho(s-1, \theta) > 0 
\qquad \mbox{for any} \quad \theta \in [0, \, \pi).   
$$
Moreover the partial derivatives of $h = h(r, \theta)$ are calculated as  
$$
\frac{\partial h}{\partial r} = \frac{s v}{r \rho^2} \quad 
\mbox{with} \quad v = v(r, \theta) := (s-r) (s-1-r) + r (2 s-1)(1+\cos \theta) ; 
\quad \frac{\partial h}{\partial \theta} = - \frac{r s \sin \theta}{\rho^2}.  
$$
This implies $(\partial h/ \partial r) (r, \theta) > 0$ for any 
$(r, \theta) \in (0, \, s-1) \times [0, \, \pi]$. 
Thus for each $\theta \in [0, \, \pi)$ there exists a unique solution 
$r = r(\theta) \in (0, \, s-1)$ to equation \eqref{eqn:re}.  
By the implicit function theorem,   
$r'(\theta) = r^2 (\sin \theta)/v > 0$ for any $\theta \in (0, \, \pi)$,  
so the function $r(\theta)$ is strictly increasing in $\theta \in [0, \, \pi)$ 
and has a limit $r_* \in (0, \, \, s-1]$ as $\theta \to \pi-0$. 
Here we have $r_* = s-1$, since 
$h(r_*, \pi) = h(s-1, \pi) = 0 > h(r, \pi)$ for any $r \in (0, \, s-1)$. 
\hfill $\Box$ \par\medskip
The function $r(\theta)$ in Lemma \ref{lem:radial} attains its minimum 
at $\theta = 0$. 
We give a lower bound for the value $r(0)$. 
\begin{lemma} \label{lem:radial0} 
For each $s \ge 26$ the value $r(0)$ admits a lower bound   
\begin{equation} \label{eqn:r(0)}
r(0) > R := \dfrac{s-1}{1 + \dfrac{1 + \log(s-1)}{s-1}}.  
\end{equation}
\end{lemma} 
{\it Proof}. 
Since $h(r, 0)$ is strictly increasing in $r \in (0, \, s-1)$ and 
$h(r(0), 0) = 0$, it suffices to show $h(R, 0) < 0$. 
Let $L(y)$ be the $C^{\infty}$-function in $y > -1$ such that 
$\log(1+ y) = y - y^2 L(y)$. 
Note that $L(0) = 1/2$, $L'(y) < 0$ and hence $L(y) < 1/2$ for any $y > 0$. 
We put $t := s-1$ for simplicity of notation. 
A careful calculation shows 
\begin{align*} 
h(R, 0) 
&= \log (2 t) + \log \left[ 1 + \frac{1}{2 t} - \frac{1 + \log t}{2(t + 1 + \log t)} \right] 
- t \log \left( 1 + \frac{1 + \log t}{t}\right) \\
& = \log (2 t) + \log \left[ 1 + \frac{1}{2 t} - \frac{1 + \log t}{2(t + 1 + \log t)} \right]  
- t \left(\frac{1 + \log t}{t}\right) + t \left(\frac{1 + \log t}{t}\right)^2 L \left( \frac{1 + \log t}{t}\right) \\
&= \log 2 - 1 + \log \left[ 1 + \frac{1}{2 t} - \frac{1 + \log t}{2(t + 1 + \log t)} \right] + 
\frac{(1 + \log t)^2}{t} L \left( \frac{1 + \log t}{t}\right) \\
&< \log 2 - 1 + \frac{1}{2 t} - \frac{1 + \log t}{2(t + 1 + \log t)} + 
\frac{(1 + \log t)^2}{2 t} =: \eta(t),  
\end{align*}
where the inequalities \eqref{eqn:log} and $L(y) < 1/2$ for $y > 0$ are used in the last line.  
Since  
$$
\eta'(t) = \frac{\log t}{2(t + 1 + \log t)^2} - \frac{(\log t)^2}{2 t^2} 
< \frac{\log t}{2 t^2} - \frac{(\log t)^2}{2 t^2} < 0 
\qquad \mbox{for} \quad t > \re,  
$$
and $\eta(25) < -0.003$, we have $\eta(t) < 0$ for any $t \ge 25$, that is, 
for any $s \ge 26$.  \hfill $\Box$ 
\begin{remark} \label{rem:radial0} 
For later use we also give a lower bound $r(0) >1$, which is much 
weaker than \eqref{eqn:r(0)} but valid for every $s > 3$. 
This follows from the fact that $\psi_s(1) = s+1-(s-1)^{s-1} < 0$ for every 
$s > 3$.  
\end{remark}
\par
The number of real roots of equation \eqref{eqn:algeq}, counted 
with multiplicity, is $2$ when $s$ is even and $3$ when $s$ is odd, respectively. 
If we put $m := \lfloor s/2 \rfloor - 1$, then $s = 2 m+ 2$ when $s$ is even 
and $s = 2 m + 3$ when $s$ is odd. 
In either case equation \eqref{eqn:algeq} has $2 m$ non-real roots, half 
of which are in the upper half-plane, while the rest are their complex 
conjugates, lying in the lower half-plane. 
In order for $m \ge 1$ we suppose $s \ge 4$.      
\begin{proposition} \label{prop:nrr} 
Let $s \ge 4$ and put $m := \lfloor s/2 \rfloor - 1 \ge 1$.  
For each $j = 1, \dots, m$, the equation  
\begin{equation} \label{eqn:Gs}
G_s(\theta) := (\sin \theta) \cdot s^s (-1)^{s-1} \sin^{s-1} [ (s-1) \theta ]  
- (s-1)^{s-1} \sin^s (s \theta) = 0    
\end{equation}
has a unique solution $\theta_j$ in the interval $I_j := 
(2 \pi j/s, \,\, 2 \pi j/(s-1))$. 
The numbers  
\begin{equation} \label{eqn:rj} 
\alpha_j := r_j \re^{\ri \theta_j} \qquad \mbox{with} \quad  
r_j := - \frac{s \sin [(s-1) \theta_j] }{\sin (s \theta_j)}, \quad   
j = 1, \dots, m,   
\end{equation}
provide all roots of equation \eqref{eqn:algeq} in the upper half-plane.  
The numbers $\rho_j := |s + \alpha_j |$ can be expressed as 
\begin{equation} \label{eqn:salpj} 
\rho_j = \frac{s \sin \theta_j}{\sin(s \theta_j)}, \qquad j = 1, \dots, m.  
\end{equation}
Let $R$ be the number defined in \eqref{eqn:r(0)}, $r_0 := r(0)$ be 
the positive root of $\psi_s(z)$ and $\rho_0 := s + r_0$. 
Then we have   
\begin{equation} \label{eqn:order} 
R < r_0 < r_1 < \cdots < r_m < s-1, \qquad \rho_0 > \rho_1 > \cdots > \rho_m > 1.  
\end{equation}
\end{proposition} 
{\it Proof}. 
Let $\alpha = r \re^{\ri \theta}$ be a root of equation \eqref{eqn:algeq} with 
$r > 0$ and $0 < \theta < \pi$ in polar expression.   
The real and imaginary parts of the equation  
$\re^{- \ri s \theta} \psi_s(\alpha) = r^{s-1} (r + s \re^{- \ri \theta}) 
- \re^{- \ri s \theta} (s-1)^{s-1} = 0$ read  
$$
r^{s-1} (s \cos \theta + r) = (s-1)^{s-1} \cos(s \theta),  \qquad 
r^{s-1} s \sin \theta = (s-1)^{s-1} \sin(s \theta),   
$$
respectively. 
The second of these equations implies $\sin(s \theta) > 0$ and 
\begin{equation} \label{eqn:salp}
|s + \alpha| = \frac{(s-1)^{s-1}}{r^{s-1}} = \frac{s \sin \theta}{\sin(s \theta)}. 
\end{equation} 
Solving the first equation for $r$ and using formula \eqref{eqn:salp} and 
the addition formula for the sine function, we have     
\begin{equation} \label{eqn:r}
r = \frac{(s-1)^{s-1}}{r^{s-1}} \cos(s \theta) - s \cos \theta 
= - \frac{ s \sin [ (s-1) \theta] }{\sin (s \theta)}, 
\end{equation}
which in particular implies $\sin [ (s-1) \theta ] < 0$. 
Substituting formula \eqref{eqn:r} into equation \eqref{eqn:salp}, we see that 
$\theta$ is a solution to equation \eqref{eqn:Gs}. 
Sign change of the sine function implies that $\theta \in (0, \, \pi)$ satisfies 
$\sin (s \theta) > 0$ and $\sin [ (s-1) \theta] < 0$, if and only if 
$\theta \in I' := \bigcup_{j=1}^{m'} I_j$ with $m' := \lfloor (s-1)/2 \rfloor$.  
We remark that $I_j$, $j = 1, \dots, m'$, are mutually disjoint, strictly increasing  
subintervals of $(0, \, \pi)$. 
Conversely, if $\theta \in I'$ is a solution to equation \eqref{eqn:Gs}, 
then $\alpha := r \re^{\ri \theta}$ with $r$ defined by formula \eqref{eqn:r} 
gives a root of equation \eqref{eqn:algeq} in the upper half-plane. 
Equation \eqref{eqn:Gs} may be thought of as the angular component 
of equation \eqref{eqn:algeq}, 
while \eqref{eqn:abs} is its radial component.     
\par
Notice that $m = m'$ if $s$ is even, while $m = m'-1$ if $s$ is odd, and 
that $I := \bigcup_{j = 1}^m I_j$ is relatively compact in the interval $(0, \, \pi)$.  
For each $j = 1, \dots, m$, at the ends of the interval $I_j$ we have          
$$
G_s(2 \pi j/s ) = s^s \sin^s( 2 \pi j/s) > 0 > 
G_s(2 \pi j/(s-1)) = - (s-1)^{s-1} \sin^s [2 \pi j/(s-1)].    
$$
Thus for each $j = 1, \dots, m$ the equation \eqref{eqn:Gs} has at least one solution 
$\theta_j \in I_j$. 
These $\theta_j$'s give rise to distinct $m$ and hence all roots $\alpha_1, \dots, \alpha_m$ 
of equation \eqref{eqn:algeq} in the upper half-plane via formula \eqref{eqn:rj}. 
In particular, $\theta_j$ is the unique solution to equation \eqref{eqn:Gs} in $I_j$. 
For $j = 1, \dots, m$, we have $0 < r_j < s-1$ by Proposition \ref{prop:s-1}
and hence $r_j = r(\theta_j)$ by Lemma \ref{lem:radial}. 
Since $r(\theta)$ is strictly increasing in $\theta$, we have 
$R < r_0 < r_1 < \cdots < r_m < s-1$, which together with $\rho_j r_j^{s-1} = (s-1)^{s-1}$ 
yields  $\rho_0 > \rho_1 > \cdots > \rho_m > 1$ in formula \eqref{eqn:salpj}. 
This shows \eqref{eqn:order}. \hfill $\Box$
\begin{lemma} \label{lem:afms} 
In the situation of Proposition $\ref{prop:nrr}$ the number $\rho_m$ 
admits the following lower bound:    
\begin{enumerate}
\setlength{\itemsep}{-1pt}
\item $\rho_m > s \sin [\pi/(s-1)] > \pi$ when $s \ge 4$ is even,   
\item $\rho_m > s \sin [2 \pi/(s-1)]$ when $s \ge 7$ is odd, while $\rho_m > 5 \sin (2 \pi/5)$ when $s = 5$.   
\end{enumerate}
\end{lemma} 
{\it Proof}. 
Formula \eqref{eqn:salpj} in Proposition \ref{prop:nrr} implies 
$\rho_m \ge s \sin \theta_m$, where $2m \pi/s < \theta_m < 2 m \pi/(s-1)$. 
For $s = 2 m + 2$ with $m \ge 1$, since $\pi/2 \le m \pi/(m+1) 
< \theta_m < 2 m \pi/(2 m + 1) < \pi$, we have $\sin \theta_m > 
\sin [2m \pi/(2m+1)] = \sin [\pi/(2 m+1)] = \sin [\pi/(s-1)]$ and hence 
$\rho_m \ge s \sin \theta_m > s \sin [ \pi/(s-1)]$. 
A little calculus shows $s \sin [ \pi/(s-1)] > \pi$ for $s \ge 4$. 
For $s = 2 m + 3$ with $m \ge 2$, since $\pi/2 \le 2 m \pi/(2 m+3) 
< \theta_m < m \pi/(m + 1) < \pi$, we have 
$\sin \theta_m > \sin[m \pi/(m+1)] = \sin [\pi/(m+1)] = \sin [2 \pi/(s-1)]$ and hence 
$\rho_m \ge s \sin \theta_m > s \sin [2 \pi/(s-1)]$. 
For $s = 5$, since $2 \pi/5 < \theta_1 < \pi/2$, we have 
$\rho_1 \ge 5 \sin \theta_1 > 5 \sin(2 \pi/5)$.   
\hfill $\Box$ 
\subsection{Arithmetic} \label{ss:arith}
Let $\alpha$ be any nontrivial root of $\psi_s(z)$ and $d := \deg \alpha$ 
its degree as an algebraic number.  
Put 
$$
M = M(\alpha) := (s + \beta_1) \cdots (s + \beta_d), \qquad 
N = N(\alpha) := \beta_1 \cdots \beta_d, 
$$
where $B = \{\beta_1, \dots, \beta_d \}$ is the set of all conjugates to $\alpha$ 
including itself.  
Since $\alpha$ is an algebraic integer, $M$ and $N$ are rational integers. 
For $s = 2$ we have $d = 2$ and $M = N = -1$. 
\begin{lemma}  \label{lem:MN}
For each integer $s \ge 3$ there exists an integer $n = n(\alpha)$ such that 
$$
|n| \ge 2, \qquad  n \, | \, (s-1), \qquad M = n^{s-1}, \qquad n N = (s-1)^d. 
$$ 
\end{lemma} 
{\it Proof}. 
For $s = 3$ we have $d = 1$, $M = 4$, $N = 1$ and $n = 2$, so the 
lemma holds true. 
Hereafter we assume $s \ge 4$. 
\par
Since $\beta_1, \dots, \beta_d$ are roots of $\psi_s(z)$, we have 
$\beta_j^{s-1} (s + \beta_j) = (s-1)^{s-1}$ for $j = 1, \dots, d$. 
Product of them gives $N^{s-1} M = (s-1)^{d(s-1)}$, namely,  
$M = \{ (s-1)^d / N \}^{s-1}$. 
Since $M$ is an integer, $n := (s-1)^d/ N$ must also be an integer, 
$M = n^{s-1}$ and $n N = (s-1)^d$. 
Let $P(z)$ be the minimal polynomial of $\alpha$.  
Since $\alpha$ is an algebraic integer, $P(z) \in \bZ[z]$ is 
a monic polynomial dividing $\psi_s(z)$ in $\bZ[z]$, so the integer $P(-s)$ 
divides $\psi_s(-s) = - (s-1)^{s-1}$ in $\bZ$. 
As $P(z) = (z-\beta_1) \cdots (z-\beta_d)$, we have  
$P(-s) = (-1)^d \cdot M$, so $M = n^{s-1}$ divides $(s-1)^{s-1}$ 
and hence $n$ divides $s-1$ in $\bZ$. 
It only remains to show $|n| \ge 2$, that is, $|n| > 1$.  
\par
Unless $s$ is even and $B$ contains the negative root 
$\beta(s) := -s -c(s)$ with $c(s)$ being given in Lemma \ref{lem:c(s)}, 
inequality \eqref{eqn:order} in Proposition \ref{prop:nrr} implies 
$|s + \beta_j| > 1$ for $j = 1, \dots, d$, and hence 
$|n|^{s-1} = |M| = |s + \beta_1| \cdots |s + \beta_d| > 1$, which yields $|n| > 1$.     
Suppose now that $s$ is even and $\beta(s) \in B$.   
We may put $\beta_1 = \beta(s)$. 
From Lemma \ref{lem:c(s)} we have $c_0 < |s + \beta_1| < 1$ and 
hence $d \ge 2$, as the algebraic integer $\beta_1$ is not a rational integer.   
If $d \ge 3$, then assertion (1) of Lemma \ref{lem:afms} and 
inequality \eqref{eqn:order} imply 
$$
|n|^{s-1} = |M| = |s + \beta_1| |s + \beta_2| \cdots |s + \beta_d| 
> c_0 \cdot \pi^{d-1} \ge c_0 \cdot \pi^2 > 2.74 > 1, 
$$
and hence $|n| > 1$.  
Finally suppose $d = 2$ and $B = \{ \beta_1, \beta_2\}$.  
Then $\beta_1 = \beta(s)$ is a real quadratic number, so its conjugate $\beta_2$ 
must be the positive root of $\psi_s(z)$ by Lemma \ref{lem:realr}. 
As $s \ge 4$, we have $|s + \beta_2| = s + \beta_2 > 4 + 1 = 5$ 
by Remark \ref{rem:radial0}.  
Therefore $|n|^{s-1} = |M| = |s + \beta_1| |s + \beta_2| > c_0 \cdot 5 > 1.39$, 
and so $|n| > 1$.  \hfill $\Box$ 
\begin{lemma} \label{lem:logk}
In the situation of Lemma $\ref{lem:MN}$, for any integer $s \ge 26$, the degree 
$d$ of $\alpha$ has a lower bound   
\begin{equation} \label{eqn:logk}
d > \frac{(\log |n|)   (s -1)}{ 1+ \log (s-1)} \ge 
\delta(s) := \frac{(\log p_s)   (s -1)}{ 1+ \log (s-1)}.  
\end{equation}
\end{lemma}
{\it Proof}. 
By inequality \eqref{eqn:order} in Proposition \ref{prop:nrr} we have 
$|\beta_j| > R$ for $j = 1, \dots, d$, so $|N| = |\beta_1 \cdots \beta_d| > R^d$ 
and 
$$
|n| = \frac{(s-1)^d}{|N|} < \frac{(s-1)^d}{R^d} = 
\left(1 + \frac{1 + \log (s-1)}{s - 1} \right)^d,   
$$
where $R$ is given in \eqref{eqn:r(0)}.  
Taking the logarithm of this estimate and using inequality \eqref{eqn:log}, we have 
$$
\log |n| < d \cdot \log \left(1 + \frac{1 + \log (s-1)}{s-1} \right) < 
d \cdot \frac{1 + \log (s-1)}{s-1}, 
$$
which together with $|n| \ge p_s$ establishes estimate \eqref{eqn:logk}. 
\hfill $\Box$ \par\medskip
{\it Proofs of Theorems $\ref{thm:deg}$ and $\ref{thm:degree}$}.  
For any $s \ge 26$ estimate \eqref{eqn:deg} follows from \eqref{eqn:logk}. 
It is easy to see that $s-2 \ge \delta(s)$ for any $s \ge 2$. 
As for $s = 2, \dots, 25$, each individual check shows that 
$\psi_s(z)$ is irreducible if $s$ is even, while the quotient of 
$\psi_s(z)$ by $ (z+s-1)^2$ is irreducible if $s$ is odd.  
In either case we have $\deg \alpha \ge s-2 \ge \delta(s)$. 
Thus estimate \eqref{eqn:deg} also holds for $s = 2, \dots, 25$, 
and Theorem \ref{thm:deg} is established.  
By Lemma \ref{lem:algequ} we have $\deg x = \deg y$. 
Theorem \ref{thm:degree} is then obtained by applying 
Theorem \ref{thm:deg} to the nontrivial root $\alpha = y$.   
\hfill $\Box$ 
\subsection{Rational and Quadratic Arguments} \label{ss:swra} 
Theorem \ref{thm:rqa} is a simple consequence of Theorems \ref{thm:a-x} 
and \ref{thm:degree}.   
We restate it in the following proposition.   
\begin{proposition} \label{prop:rqa} 
Let $\lambda = (p, 0, r; a, 1/2; x) \in \cI$ be any non-elementary solution.  
\begin{enumerate}
\setlength{\itemsep}{-1pt}
\item If $x$ is a rational number, then $\lambda$ is a multiple of 
solution \eqref{eqn:sol-r1} or solution \eqref{eqn:sol-r2} in 
Example $\ref{ex:rat-x}$.   
\item If $x$ is a quadratic irrational, then $\lambda$ is a multiple of 
solution \eqref{eqn:sol-ir} in Example $\ref{ex:irr-x}$. 
\end{enumerate}
\end{proposition}
{\it Proof}. 
 We apply Theorem \ref{thm:degree}, or rather Remark \ref{rem:deg}, to the root 
$x$ of equation \eqref{eqn:alg-eq}.   
For the functions in \eqref{eqn:delta} we observe $\delta_0(6) > 2.1$ and 
$\delta_1(11) > 2.09$. 
Thus if $\deg x \le 2$ then $s = r/p \in \{2, 4 \} \cup \{ 3, 5, 7, 9 \}$. 
Among these candidates we have $\deg x = 1$ (rational case) 
only when $s = 3$ and $\deg x = 2$ (quadratic case) only when $s = 2$. 
\par
Assertion (1). 
Suppose that $x$ is a rational number. 
Then $s = r/p = 3$, $y = 1$ and $x = 3/4$. 
Formula \eqref{eqn:saap} in Theorem \ref{thm:a-x} implies 
$j + j' = 1$ and hence either $(j, j') = (0, 1)$ or $(j, j') = (1, 0)$.   
In the former case we have $a = 0$ and $\lambda$ must be the $p$-multiple 
of $\lambda_1 = (1, \, 0, \,  3; \, 0, \, 1/2; \, 3/4)$. 
In the latter case we have $a = 1/3$ and $\lambda$ must be the $p$-multiple 
of $\lambda_2 = (1, \, 0, \, 3; \, 1/3, \, 1/2; \, 3/4)$.    
Notice that $\lambda_1$ and $\lambda_2$ are exactly 
the solutions \eqref{eqn:sol-r1} and \eqref{eqn:sol-r2} in Example \ref{ex:rat-x}. 
Thus assertion (1) follows immediately. 
\par
Assertion (2). 
Suppose that $x$ is a quadratic irrational. 
Then $s = r/p = 2$, $y = -1 + \sqrt{2}$ and $x = 2 \sqrt{2}-2$. 
Formula \eqref{eqn:saap} in Theorem \ref{thm:a-x} implies $j + j' = 0$,  
which is the case only when $(j, j') = (0, 0)$ and $a = 0$. 
Moreover $p(j + j' + 1) = p$ must be even, so $\lambda$ is the $(p/2)$-multiple of 
$\lambda_3 = (2, \, 0, \,  4; \, 0, \, 1/2; \, 2 \sqrt{2} - 2)$.     
Notice that $\lambda_3$ is exactly the solutions \eqref{eqn:sol-ir} in Example \ref{ex:irr-x}. 
Thus assertion (2) follows immediately. \hfill $\Box$ 
\par\vspace{5mm} 
{\bf Acknowledgments}. 
This work was supported by JSPS KAKENHI Grant Number JP22K03365. 
Sadly, the second author passed away while this work was in progress. 
The first author dedicates the completed version of this article to her memory.   

\end{document}